\documentclass[draft]{amsart}
\usepackage{amsmath,amsthm}
\usepackage[english]{babel}
\usepackage{amsfonts}
\usepackage{latexsym}

\newtheorem{lemma}{Lemma}
\newtheorem{theorem}{Theorem}
\newtheorem{corollary}{Corollary}
\newcommand*{\cd}{(\cdot)}
\newcommand*{\lp}{L_p(T,\mu)}
\newcommand*{\lr}{L_r(T,\mu)}
\newcommand*{\lqq}{L_q(T,\mu)}
\newcommand*{\lpn}{L_p(T_0,\mu)}
\newcommand*{\wx}{\widehat x}
\newcommand*{\ov}{\overline}
\newcommand*{\iTn}{\int_{T_0}}
\newcommand*{\iT}{\int_T}
\newcommand*{\wm}{\widehat m}
\newcommand*{\LL}{\mathcal L}
\newcommand*{\wl}{\widehat\lambda}
\newcommand*{\wu}{\widehat u}
\newcommand*{\iTnn}{\int_{T_1}}
\newcommand*{\wCC}{\widetilde C}
\newcommand*{\wk}{\widetilde k}
\newcommand*{\wtw}{\widetilde w}
\newcommand*{\wt}{\widetilde\theta}
\newcommand*{\wg}{\widetilde\gamma}
\newcommand*{\wII}{\widetilde I}
\newcommand*{\wps}{\widetilde\psi}
\newcommand*{\wva}{\widetilde\varphi}
\newcommand*{\wxi}{\widehat\xi}
\newcommand*{\wpp}{\widehat p}
\newcommand*{\wqq}{\widehat q}
\newcommand*{\vq}{\widetilde q}
\newcommand*{\wC}{\widetilde K}
\newcommand*{\wtt}{\widetilde t}
\newcommand*{\wc}{\widetilde s}
\newcommand*{\ld}{L_2(\mathbb R^d)}
\newcommand*{\lpd}{L_p(\mathbb R^d)}
\newcommand*{\wL}{\widetilde L}
\newcommand*{\wwg}{\widehat\gamma}
\newcommand*{\wwq}{\widehat q^*}
\newcommand*{\iRd}{\int_{\mathbb R^d}}

\newcommand*{\la}{\langle}
\newcommand*{\ra}{\rangle}

\DeclareMathOperator*{\vraisup}{vraisup}
\DeclareMathOperator*{\mes}{mes}
\begin{document}

\title[Optimal recovery and generalized Carlson inequality]{Optimal recovery and generalized Carlson inequality for weights with symmetry properties}
\author{K.~Yu.~Osipenko}

\address{Moscow State University, Moscow and Institute for Information Transmission Problems,
Russian Academy of Sciences, Moscow}
\email{kosipenko@yahoo.com}
\subjclass[2010]{41A65, 41A46, 49N30}
\keywords{Optimal recovery, Carlson type inequalities, Fourier transform, sharp constants}
\begin{abstract}
The paper concerns problems of the recovery of operators from noisy information in weighted $L_q$-spaces with homogeneous weights. A number of general theorems are proved and applied to finding exact constants in multidimensional Carlson type inequalities with several weights and problems of the recovery of differential operators from a noisy Fourier transform. In particular, optimal methods are obtained for the recovery of powers of generalized Laplace operators from a noisy Fourier transform in the $L_p$-metric.
\end{abstract}

\maketitle

\section{Introduction}

Let $T$ be a nonempty set, $\Sigma$ be the $\sigma$-algebra of subsets of $T$, and $\mu$ be a nonnegative $\sigma$-additive measure on $\Sigma$. We denote by $L_p(T,\Sigma,\mu)$ (or simply $\lp$) the set of all $\Sigma$-measurable functions with values in $\mathbb R$ or in $\mathbb C$ for which
\begin{align*}
\|x\cd\|_{\lp}&=\biggl(\int_T|x(t)|^p\,d\mu(t)\biggr)^{1/p}<\infty,\quad1\le p<\infty,\\
\|x\cd\|_{L_\infty(T,\mu)}&=\vraisup_{t\in T}|x(t)|<\infty,\quad p=\infty.
\end{align*}
If $T\subset\mathbb R^d$ and $d\mu=dt$, $t\in\mathbb R^d$, we put $L_p(T)=\lp$.

The Carlson inequality \cite{Ca}
\begin{equation*}
\|x(t)\|_{L_1(\mathbb R_+)}\le\sqrt\pi\|x(t)\|_{L_2(\mathbb R_+)}^{1/2}
\|tx(t)\|_{L_2(\mathbb R_+)}^{1/2},\quad\mathbb R_+=[0,+\infty),
\end{equation*}
was generalized by many authors (see \cite{Le}, \cite{An}, \cite{B}, \cite{Os}, \cite{Os21}). In \cite{Os} we found sharp constants for inequalities of the form
\begin{equation*}
\|w\cd x\cd\|_{\lqq}\le K\|w_0\cd x\cd\|_{\lp)}^\gamma\|w_1\cd x\cd\|_{\lr}^{1-\gamma},
\end{equation*}
where $T$ is a cone in a linear space, $w\cd$, $w_0\cd$, and $w_1\cd$ are homogenous functions and $1\le q<p,r<\infty$ (for $T=\mathbb R^d$ the sharp inequality was obtained in \cite{B}). This problem is closely related with the following extremal problem
$$\|w\cd x\cd\|_{\lqq}\to\max,\quad\|w_0\cd x\cd\|_{\lp)}\le\delta,\quad
\|w_1\cd x\cd\|_{\lr}\le1,$$
where $\delta>0$. In this paper we study the extremal problem
\begin{multline}\label{ex1}
\|w\cd x\cd\|_{\lqq}\to\max,\quad\|w_0\cd x\cd\|_{\lp)}\le\delta,\\\|w_j\cd x\cd\|_{\lr}\le1,\ j=1,\ldots,n,
\end{multline}
where $w\cd$, $w_0\cd$, and $w_j\cd$, $j=1,\ldots,n$, are homogenous functions with some symmetry properties. Using the solution of this problem we obtain the sharp constant $K$ for the inequality
\begin{equation*}
\|w\cd x\cd\|_{\lqq}\le K\|w_0\cd x\cd\|_{\lp)}^\gamma\left(\max_{1\le j\le n}\|\omega_j\cd x\cd\|_{\lr}\right)^{1-\gamma}.
\end{equation*}
In particular, we find the sharp constant for the inequality
\begin{equation*}
\|w\cd x\cd\|_{L_q(\mathbb R_+^d)}\le C\|w_0\cd x\cd\|_{L_p(\mathbb R_+^d)}^{p\alpha}\left(\max_{1\le j\le d}\|\omega_j\cd x\cd\|_{L_r(\mathbb R_+^d)}\right)^{r\beta},
\end{equation*}
where $w(t)=(t_1^2+\ldots+t_d^2)^{\theta/2}$, $w_0(t)=(t_1^2+\ldots+t_d^2)^{\theta_0/2}$,
$w_j(t)=t_j^{\theta_1}$, $j=1,\ldots,d$, $\theta=d(1-1/q)$, $\theta_0=d-(\lambda+d)/p$, $\theta_1=d+(\mu-d)/r$,
\begin{equation*}
\alpha=\frac\mu{p\mu+r\lambda},\quad\beta=\frac\lambda{p\mu+r\lambda},\quad\lambda,\mu>0,
\end{equation*}
and $(p,q,r)\in P\cup P_1\cup P_2$, where
\begin{multline*}
P=\{\,(p,q,r):1\le q<p,r\,\},\quad P_1=\{\,(p,q,r):1\le q=r<p\,\},\\ P_2=\{\,(p,q,r):1\le q=p<r\,\}.
\end{multline*}
For $d=1$, $q=1$, and $(p,1,r)\in P$ this result was proved in \cite{Le} (see also \cite{B}).

It is appeared that the value of \eqref{ex1} is the error of optimal recovery of the operator $\Lambda x\cd=w\cd x\cd$ on the class of functions $x\cd$ such that $\|w_j\cd x\cd\|_{\lr}\le1$, $j=1,\ldots,n$, by the information about the function $w_0\cd x\cd$ given with the error $\delta$ in $L_p$-norm. Therefore, in section \ref{main} we begin with the setting of optimal recovery problem  and then in section \ref{res} we prove some general theorems. In section \ref{hom} we consider the case when weights are homogeneous in a cone of linear space and section \ref{rd} is devoted to the case of $\mathbb R^d$. In section \ref{dif} the results obtained are applied to optimal recovery and sharp inequalities of differential operators defined by Fourier transforms.

\section{General setting}\label{main}

Let $T_0$ is not empty $\mu$-measurable subset of $T$. Put
\begin{gather*}
\mathcal W=\{\,x\cd:x\cd\in\lpn,\ \|\varphi_j\cd x\cd\|_{\lr}<\infty,\ j=1,\ldots,n\,\},\\
W=\{\,x\cd\in\mathcal W:\|\varphi_j\cd x\cd\|_{\lr}\le1,\ j=1,\ldots,n\,\},
\end{gather*}
where $1\le p,r\le\infty$, and $\varphi_j\cd$ is a measurable function on $T$. Consider the problem of recovery of operator $\Lambda\colon\mathcal W\to L_q(T,\mu)$, $1\le q\le\infty$, defined by equality $\Lambda x\cd=\psi\cd x\cd$, where $\psi\cd$ is a measurable function on $T$, on the class $W$ by the information about functions $x\cd\in W$ given inaccurately (we assume that $\psi\cd$ and $\varphi_j\cd$, $j=1,\ldots,n$, such that $\Lambda$ maps $\mathcal W$ to $\lqq$). More precisely, we assume that for any function $x\cd\in W$ we know $y\cd\in\lpn$ such that $\|x\cd-y\cd\|_{\lpn}\le\delta$, $\delta>0$. We want to approximate the value $\Lambda x\cd$ knowing $y\cd$. As recovery methods we consider all possible mappings $m\colon\lpn\to\lqq$.
The error of a~method~$m$ is defined as
\begin{equation*}
e(p,q,r,m)=\sup_{\substack{x\cd\in W,\ y\cd\in\lpn\\\|x\cd-y\cd\|_{\lpn}\le\delta}}\|\Lambda x\cd-m(y)\cd\|_{\lqq}.
\end{equation*}

The quantity
\begin{equation}\label{Ep}
E(p,q,r)=\inf_{m\colon\lpn\to\lqq}e(p,q,r,m)
\end{equation}
is known as the optimal recovery error, and a method on which this infimum is attained is called optimal. Various settings of optimal recovery theory and examples of such problems may be found in \cite{MR}, \cite{TW}, \cite{Pl}, \cite{Osb}, \cite{Osb22}.

For the lower bound of $E(p,q,r)$ we use the following result which was proved (in more or less general forms) in many papers (see, for example, \cite{Os14}).

\begin{lemma}\label{L1}
\begin{equation}\label{lb}
E(p,q,r)\ge\sup_{\substack{x\cd\in W\\\|x\cd\|_{\lpn}\le\delta}}\|\Lambda x\cd\|_{\lqq}.
\end{equation}
\end{lemma}

\section{Main results}\label{res}

Set
\begin{equation*}
\chi_0(t)=\begin{cases}
1,&t\in T_0,
\\
0,&t\notin T_0,
\end{cases}\quad\sigma_r(t)=\sum_{j=1}^n\lambda_j|\varphi_j(t)|^r.
\end{equation*}

\begin{theorem}\label{T1}
Let $1\le q<p,r$, $\lambda_j\ge0$, $j=0,1,\ldots,n$, $\lambda_0+\sigma_r(t)\ne0$ for almost all $t\in T_0$, $\sigma_r(t)\ne0$ for almost all $t\in T\setminus T_0$, $\wx(t)\ge0$ be a solution of equation
\begin{equation}\label{ko1}
-q|\psi(t)|^q+p\lambda_0x^{p-q}(t)\chi_0(t)+r\sigma_r(t)x^{r-q}(t)=0,
\end{equation}
$\ov\lambda$ such that
\begin{multline}\label{nez}
\iTn\wx^p(t)\,d\mu(t)\le\delta^p,\quad\iT|\varphi_j(t)|^r\wx^r(t)\,d\mu(t)\le1,\ j=1,\ldots,n,\\
\lambda_0\biggl(\iTn\wx^p(t)\,d\mu(t)-\delta^p\biggr)=0,\quad
\lambda_j\biggr(\iT|\varphi_j(t)|^r\wx^r(t)\,d\mu(t)-1\biggr)=0,\ j=1,\ldots,n.
\end{multline}
Then
\begin{equation}\label{Epqr}
E(p,q,r)=\biggl(q^{-1}p\lambda_0\delta^p+q^{-1}r\sum_{j=1}^n\lambda_j\biggr)^{1/q},
\end{equation}
and the method
\begin{equation}\label{met}
\wm(y)(t)=\begin{cases}q^{-1}p\lambda_0\wx^{p-q}(t)|\psi(t)|^{-q}
\psi(t)y(t),&t\in T_0,\ \psi(t)\ne0,\\
0,&\mbox{otherwise},\end{cases}
\end{equation}
is optimal recovery method.
\end{theorem}

To prove this theorem we need some preliminary results. The first one is actually a sufficient condition in the Kuhn-Tucker theorem (the only difference is that we do not require convexity of functions).

Let $f_j\colon A\to\mathbb R$, $j=0,1,\ldots,k$, be functions defined on some set $A$. Consider the extremal problem
\begin{equation}\label{ex0}
f_0(x)\to\max,\quad f_j(x)\le0,\quad j=1,\ldots,k,\quad x\in A,
\end{equation}
and write down its Lagrange function
\begin{equation*}
\LL(x,\lambda)=-f_0(x)+\sum_{j=1}^k\lambda_jf_j(x),\quad
\lambda=(\lambda_1,\ldots,\lambda_k).
\end{equation*}

\begin{lemma}\label{L2}
Assume that there exist $\wl_j\ge0$, $j=1,\ldots,k$, and an element $\wx\in A$,
admissible for problem \eqref{ex0}, such that
\begin{align*}
(a)&\quad\min_{x\in A}\LL(x,\wl)=\LL(\wx,\wl),\quad\wl=(\wl_1,\ldots,\wl_k),\\
(b)&\quad\wl_jf_j(\wx)=0,\ j=1,\ldots,k.
\end{align*}
Then $\wx$ is an extremal element for problem \eqref{ex0}.
\end{lemma}

\begin{proof}
For any $x$ admissible for problem \eqref{ex0} we have
\begin{equation*}
-f_0(x)\ge\LL(x,\wl)\ge\LL(\wx,\wl)=-f_0(\wx).
\end{equation*}
\end{proof}

Put
$$F(u,v,\alpha)=-((1-\alpha)u+\alpha v)^q+av^p+bu^r,\quad u,v\ge0,\quad\alpha\in[0,1],$$
where $a,b\ge0$, and $1\le p,q,r<\infty$.

\begin{lemma}[\cite{Os}]\label{L3}
For all $a,b\ge0$, $a+b>0$, and all $1\le q<p,r<\infty$, there exists the unique solution $\wu>0$ of the equation
\begin{equation*}
-q+pau^{p-q}+rbu^{r-q}=0.
\end{equation*}
Moreover, for all $u,v\ge0$ and $\alpha=q^{-1}pa\wu^{p-q}=1-q^{-1}rb\wu^{r-q}$
\begin{equation*}
F(\wu,\wu,\alpha)\le F(u,v,\alpha).
\end{equation*}
In particular, for all $u\ge0$
\begin{equation*}
-\wu^q+a\wu^p+b\wu^r\le-u^q+au^p+bu^r.
\end{equation*}
\end{lemma}

\begin{proof}[Proof of Theorem~\ref{T1}]

1. Lower estimate. The extremal problem on the right-hand side of \eqref{lb} (for convenience, we raise the quantity to be maximized to the $q$-th power) is as follows:
\begin{multline}\label{ex}
\iT|\psi(t)x(t)|^q\,d\mu(t)\to\max,\quad\iTn|x(t)|^p\,d\mu(t)\le\delta^p,\\
\iT|\varphi_j(t)x(t)|^r\,d\mu(t)\le1,\ j=1,\ldots,n.
\end{multline}
If $t\in T$ such that $\psi(t)=0$, then evidently $\wx(t)=0$. If $\psi(t)\ne0$ we obtain by Lemma~\ref{L3} that that there is the unique solution $\wx(t)$ of \eqref{ko1}. It follows by \eqref{nez} that $\wx\cd$ is admissible function for problem \eqref{ex}. Therefore, by \eqref{lb} we obtain
\begin{equation*}
E(p,q,r)\ge\biggl(\iT|\psi(t)|^q\wx^q(t)\,d\mu(t)\biggr)^{1/q}.
\end{equation*}
From \eqref{ko1} we have
\begin{equation*}
|\psi(t)|^q\wx^q(t)=q^{-1}p\lambda_0\wx^p(t)\chi_0(t)+q^{-1}r\sigma_r(t)\wx^r(t).
\end{equation*}
Integrating this equality over the set $T$, we obtain
\begin{equation*}
\iT|\psi(t)|^q\wx^q(t)\,d\mu(t)=q^{-1}p\lambda_0\delta^p+q^{-1}r\sum_{j=1}^n\lambda_j.
\end{equation*}
Thus,
\begin{equation*}
E(p,q,r)\ge\biggl(q^{-1}p\lambda_0\delta^p+q^{-1}r\sum_{j=1}^n\lambda_j\biggr)^{1/q}.
\end{equation*}

2. Upper estimate. To estimate the error of method \eqref{met} we need to find the value of the extremal problem:
\begin{multline}\label{La}
\iTn|\psi(t)x(t)-\psi(t)\alpha(t)y(t)|^q\,d\mu(t)+\int_{T\setminus T_0}|\psi(t)x(t)|^q\,d\mu(t)\to\max,\\\iTn|x(t)-y(t)|^p\,d\mu(t)\le
\delta^p,\quad\iT|\varphi_j(t)x(t)|^r\,d\mu(t)\le1,\ j=1,\ldots,n,
\end{multline}
where
\begin{equation*}
\alpha(t)=\begin{cases}q^{-1}p\lambda_0\wx^{p-q}(t)|\psi(t)|^{-q},&t\in T_0,\ \psi(t)\ne0,\\
0,&\mbox{otherwise}.\end{cases}
\end{equation*}
Put
\begin{equation*}
z(t)=\begin{cases}x(t)-y(t),&t\in T_0,\\
0,&t\in T\setminus T_0.\end{cases}
\end{equation*}
Then \eqref{La} may be rewritten as follows:
\begin{multline*}
\iT|\psi(t)|^q|(1-\alpha(t))x(t)+\alpha(t)z(t)|^q\,d\mu(t)\to\max,\\
\iTn|z(t)|^p\,d\mu(t)\le\delta^p,\quad\iT|\varphi_j(t)x(t)|^r\,d\mu(t)\le1,\ j=1,\ldots,n.
\end{multline*}
The value of this problem does not exceed the value of the problem
\begin{multline}\label{uv}
\iT|\psi(t)|^q((1-\alpha(t))u(t)+\alpha(t)v(t))^q\,d\mu(t)\to\max,\\\iTn v^p(t)\,d\mu(t)\le\delta^p,\quad
\iT|\varphi_j(t)|^ru^r(t)\,d\mu(t)\le1,\ j=1,\ldots,n,\\ u(t)\ge0,\ v(t)\ge0\quad\mbox{for almost all }\ t\in T.
\end{multline}
The Lagrange function for this problem is
\begin{equation*}
\LL(u\cd,v\cd,\ov\lambda)=\iT L(t,u(t),v(t),\ov\lambda)\,d\mu(t),
\end{equation*}
where
\begin{equation*}
L(t,u,v,\ov\lambda)=-|\psi(t)|^q((1-\alpha(t))u+\alpha(t)v)^q
+\lambda_0v^p\chi_0(t)+\sigma_r(t)u^r.
\end{equation*}
By Lemma~\ref{L3} we have
\begin{equation*}
L(t,\wx(t),\wx(t),\ov\lambda)\le L(t,u(t),v(t),\ov\lambda).
\end{equation*}
Thus,
\begin{equation*}
\LL(\wx\cd,\wx\cd,\ov\lambda)\le\LL(u\cd,v\cd,\ov\lambda).
\end{equation*}
It follows by Lemma~\ref{L2} that functions $u\cd=v\cd=\wx\cd$ are extremal in \eqref{uv}. Consequently,
$$e(p,q,r,\wm)\le\biggl(\iT|\psi(t)|^q\wx^q(t)\,d\mu(t)\biggr)^{1/q}
=\biggl(q^{-1}p\lambda_0\delta^p+q^{-1}r\sum_{j=1}^n\lambda_j\biggr)^{1/q}\le E(p,q,r).$$
It means that method \eqref{met} is optimal and equality \eqref{Epqr} holds.
\end{proof}

Denote $a_+=\max\{a,0\}$.

\begin{theorem}\label{T2}
Let $1\le q=r<p$, $\lambda_0>0$, $\lambda_j\ge0$, $j=1,\ldots,n$,
\begin{equation}\label{xT2}
\wx(t)=\begin{cases}\displaystyle\left(\dfrac q{p\lambda_0}\left(|\psi(t)|^q-\sigma_q(t)
\right)_+\right)^{\frac1{p-q}},&t\in T_0,\\
0,&t\notin T_0,\end{cases}
\end{equation}
$\ov\lambda$ satisfies conditions \eqref{nez}, and $|\psi(t)|^q-\sigma_q(t)\le0$ for almost all $t\notin T_0$. Then
\begin{equation}\label{Eq=r}
E(p,q,q)=\biggl(q^{-1}p\lambda_0\delta^p+\sum_{j=1}^n\lambda_j\biggr)^{1/q},
\end{equation}
and the method
\begin{equation}\label{met1}
\wm(y)(t)=\begin{cases}\left(1-|\psi(t)|^{-q}\sigma_q(t)\right)_+\psi(t)y(t),&t\in T_0,\ \psi(t)\ne0,\\
0,&\mbox{othervise},\end{cases}
\end{equation}
is optimal.
\end{theorem}

\begin{proof}

1. Lower estimate. It follows by \eqref{nez} that $\wx\cd$ is admissible function for extremal problem in the right-hand side of \eqref{lb}. Therefore,
\begin{equation*}
E(p,q,q)\ge\biggl(\iT|\psi(t)|^q\wx^q(t)\,d\mu(t)\biggr)^{1/q}.
\end{equation*}
From the definition of $\wx\cd$ we have
\begin{equation*}
|\psi(t)|^q\wx^q(t)=q^{-1}p\lambda_0\wx^p(t)\chi_0(t)+\sigma_q(t)\wx^q(t).
\end{equation*}
Integrating this equality, we obtain
\begin{equation*}
\iT|\psi(t)|^q\wx^q(t)\,d\mu(t)=q^{-1}p\lambda_0\delta^p+\sum_{j=1}^n\lambda_j.
\end{equation*}
Thus,
\begin{equation*}
E(p,q,q)\ge\biggl(q^{-1}p\lambda_0\delta^p+\sum_{j=1}^n\lambda_j\biggr)^{1/q}.
\end{equation*}

2. Upper estimate. Put
\begin{equation*}
\alpha(t)=\begin{cases}\left(1-|\psi(t)|^{-q}\sigma_q(t)\right)_+ ,&t\in T_0,\ \psi(t)\ne0,\\
0,&\mbox{othervise},\end{cases}
\end{equation*}
To estimate the error of method \eqref{met1} we need to find the value of the extremal problem:
\begin{multline*}
\iTn|\psi(t)|^q|x(t)-\alpha(t)y(t)|^q\,d\mu(t)+\int_{T\setminus T_0}|\psi(t)x(t)|^q\,d\mu(t)\to\max,\\
\iTn|x(t)-y(t)|^p\,d\mu(t)\le\delta^p,\quad\iT|\varphi_j(t)x(t)|^q\,d\mu(t)\le1,\ j=1,\ldots,n.
\end{multline*}
Putting $z\cd=x\cd-y\cd$ this problem may be rewritten in the following form
\begin{multline*}
\iTn|\psi(t)|^q|(1-\alpha(t))x(t)+\alpha(t)z(t)|^q\,d\mu(t)
+\int_{T\setminus T_0}|\psi(t)x(t)|^q\,d\mu(t)\to\max,\\
\iTn|z(t)|^p\,d\mu(t)\le\delta^p,\quad\iT|\varphi_j(t)x(t)|^q\,d\mu(t)\le1,\ j=1,\ldots,n.
\end{multline*}
The value of this problem evidently coincides with the value of the problem
\begin{multline}\label{uv1}
\iT|\psi(t)|^q((1-\alpha(t))v(t)+\alpha(t)u(t))^q\,d\mu(t)\to\max,\\
\iTn u^p(t)\,d\mu(t)\le\delta^p,\quad\iT|\varphi_j(t)|^qv^q(t)\,d\mu(t)\le1,\ j=1,\ldots,n,\\ u(t),v(t)\ge0,\ \mbox{for almost all }t\in T.
\end{multline}

The Lagrange function of \eqref{uv1} has the form
\begin{equation*}
\LL(u\cd,v\cd,\ov\lambda)=\iT L(u(t),v(t),\ov\lambda)\,d\mu(t),
\end{equation*}
where
\begin{equation*}
L(u,v,\ov\lambda)=\begin{cases}-|\psi(t)|^q((1-\alpha(t))v+\alpha(t)u)^q+\lambda_0u^p+
\sigma_q(t)v^q,&t\in T_0,\\
-|\psi(t)|^qv^q+\sigma_q(t)v^q,&t\notin T_0.\end{cases}
\end{equation*}
If $\alpha(t)>0$, then
\begin{equation*}
\frac{\partial L}{\partial v}=q(v^{q-1}-((1-\alpha(t))v+\alpha(t)u)^{q-1})\sigma_q(t).
\end{equation*}
Therefore, for $\alpha(t)>0$ and any $u>0$, the function $L(u,v,\ov\lambda)$, $v\in(0,+\infty)$, reaches a minimum at $v=u$. Set $T_0'=\{\,t\in T_0:\alpha(t)>0\,\}$. We have
\begin{equation*}
\LL(u\cd,v\cd,\ov\lambda)\ge\int_{T_0'}L(u\cd,u\cd,\ov\lambda)\,d\mu(t).
\end{equation*}
It is easily checked that for $t\in T_0'$ for all $u(t)\ge0$
\begin{equation*}
L(u\cd,u\cd,\ov\lambda)\ge L(\wx\cd,\wx\cd,\ov\lambda).
\end{equation*}
Consequently,
\begin{equation*}
\LL(u\cd,v\cd,\ov\lambda)\ge\int_{T_0'}L(\wx\cd,\wx\cd,\ov\lambda)\,d\mu(t)=
\LL(\wx\cd,\wx\cd,\ov\lambda).
\end{equation*}
Taking into account \eqref{nez} we obtain by Lemma \ref{L2} that $u\cd=v\cd=\wx\cd$ are extremal functions in \eqref{uv1}. Thus,
\begin{equation*}
e^q(p,q,q,\wm)=\iT|\psi(t)\wx(t)|^q\,d\mu(t)=q^{-1}p\lambda_0\delta^p+\sum_{j=1}^n\lambda_j\le E^q(p,q,q).
\end{equation*}
It means that the method $\wm$ is optimal and the optimal recovery error is as stated.
\end{proof}

\begin{theorem}\label{TT3}
Let $1\le q=p<r$, $\lambda_0>0$, $\lambda_j\ge0$, $j=1,\ldots,n$,
$\sigma_r(t)\ne0$ for almost all $t\in T$,
\begin{equation}\label{xT3}
\wx(t)=\begin{cases}\left(pr^{-1}\sigma_r^{-1}(t)(|\psi(t)|^p-\lambda_0)_+\right)^{\frac1{r-p}},
&t\in T_0,\\
\left(pr^{-1}\sigma_r^{-1}(t)|\psi(t)|^p\right)^{\frac1{r-p}},&t\in T \setminus T_0,\end{cases}
\end{equation}
and $\ov\lambda$ satisfies conditions \eqref{nez}. Then
\begin{equation}\label{Ep=q}
E(p,p,r)=\biggl(\lambda_0\delta^p+\frac rp\sum_{j=1}^n\lambda_j\biggr)^{1/p},
\end{equation}
and the method
\begin{equation}\label{met11}
\wm(y)(t)=\begin{cases}\alpha(t)\psi(t)y(t),&t\in T_0,\\
0,&t\in T\setminus T_0,\end{cases}\end{equation}
where
\begin{equation*}
\alpha(t)=\begin{cases}\min\left\{1,\lambda_0|\psi(t)|^{-p}\right\},&t\in T_0,\ \psi(t)\ne0,\\
0,&\mbox{othervise},\end{cases}
\end{equation*}
is optimal.
\end{theorem}

\begin{proof}
1. Lower estimate. By the definition of $\wx\cd$ we have
\begin{equation*}
|\psi(t)|^p\wx^p(t)=\lambda_0\wx^p(t)\chi_0(t)+\frac rp\sigma_r(t)\wx^r(t).
\end{equation*}
Using the similar arguments as in the proof of Theorem \ref{T1} we obtain
\begin{equation*}
E(p,p,r)\ge\biggl(\iT|\psi(t)|^p\wx^p(t)\,d\mu(t)\biggr)^{1/p}=\biggl(\lambda_0\delta^p+\frac rp\sum_{j=1}^n\lambda_j\biggr)^{1/p}.
\end{equation*}

2. Upper estimate. To estimate the error of method \eqref{met11} we need to find the value of the following extremal problem:
\begin{multline*}
\iTn|\psi(t)|^p|x(t)-\alpha(t)y(t)|^p\,d\mu(t)+\int_{T\setminus T_0}|\psi(t)x(t)|^p\,d\mu(t)\to\max,\\
\iTn|x(t)-y(t)|^p\,d\mu(t)\le\delta^p,\quad\iT|\varphi_j(t)x(t)|^r\,d\mu(t)\le1,\ j=1,\ldots,n.
\end{multline*}
Putting $z\cd=x\cd-y\cd$ this problem may be rewritten in the form
\begin{multline*}
\iTn|\psi(t)|^p|(1-\alpha(t))x(t)+\alpha(t)z(t)|^p\,d\mu(t)
+\int_{T\setminus T_0}|\psi(t)x(t)|^p\,d\mu(t)\to\max,\\
\iTn|z(t)|^p\,d\mu(t)\le\delta^p,\quad\iT|\varphi_j(t)x(t)|^r\,d\mu(t)\le1,\ j=1,\ldots,n.
\end{multline*}
The value of this problem evidently coincides with the value of the problem
\begin{multline}\label{uv11}
\iT|\psi(t)|^p((1-\alpha(t))v(t)+\alpha(t)u(t))^p\,d\mu(t)\to\max,\\
\iTn u^p(t)\,d\mu(t)\le\delta^p,\quad\iT|\varphi_j(t)|^rv^r(t)\,d\mu(t)\le1,\ j=1,\ldots,n,\\ u(t),v(t)\ge0,\ \mbox{for almost all }t\in T.
\end{multline}

The Lagrange function of \eqref{uv11} has the form
\begin{equation*}
\LL(u\cd,v\cd,\ov\lambda)=\iT L(u(t),v(t),\ov\lambda)\,d\mu(t),
\end{equation*}
where
\begin{equation*}
L(u,v,\ov\lambda)=\begin{cases}-|\psi(t)|^p((1-\alpha(t))v+\alpha(t)u)^p+\lambda_0u^p+
\sigma_r(t)v^r,&t\in T_0,\\
-|\psi(t)|^pv^p+\sigma_r(t)v^r,&t\in T\setminus T_0.\end{cases}
\end{equation*}
For $t\in T_0$ and $|\psi(t)|^p>\lambda_0$ we have
\begin{equation*}
\frac{\partial L}{\partial u}=p\lambda_0(u^{p-1}-((1-\alpha(t))v+\alpha(t)u)^{p-1}).
\end{equation*}
Consequently, in this case for any $v>0$ the function $L(u,v,\ov\lambda)$, $v\in(0,+\infty)$, reaches a minimum at $v=u$. If $t\in T_0$, $0<|\psi(t)|^p\le\lambda_0$, then $\alpha(t)=1$ and $L(u,v,\ov\lambda)\ge0$. If $t\in T_0$ and $\psi(t)=0$, then again $L(u,v,\ov\lambda)\ge0$. Set
$T_1=\{t\in T_0:|\psi(t)|^p>\lambda_0\}$. Then for all $u(t),v(t)\ge0$ we have
\begin{equation*}
\LL(u\cd,v\cd,\ov\lambda)\ge\iTnn L(v\cd,v\cd,\ov\lambda)\,d\mu(t)+\int_{T\setminus T_0}
L(v\cd,v\cd,\ov\lambda)\,d\mu(t).
\end{equation*}
It is easy to check that for all $v(t)\ge0$
\begin{equation*}
L(v\cd,v\cd,\ov\lambda)\ge L(\wx\cd,\wx\cd,\ov\lambda).
\end{equation*}
Therefore,
\begin{equation*}
\LL(u\cd,v\cd,\ov\lambda)\ge\int_{T_1\cup(T\setminus T_0)}L(\wx\cd,\wx\cd,\ov\lambda)\,d\mu(t)=
\LL(\wx\cd,\wx\cd,\ov\lambda).
\end{equation*}
Taking into account \eqref{nez} we obtain by Lemma \ref{L2} that $u\cd=v\cd=\wx\cd$ are extremal functions in \eqref{uv11}. Consequently,
\begin{equation*}
e^p(p,p,r,\wm)=\iT|\psi(t)\wx(t)|^q\,d\mu(t)=\lambda_0\delta^p+\frac rp\sum_{j=1}^n\lambda_j\le E^p(p,p,r).
\end{equation*}
It means that the method $\wm$ is optimal and the optimal recovery error is as stated.
\end{proof}

Note that if conditions of Theorems \ref{T1}, \ref{T2}, and \ref{TT3} are fulfilled, then we have
\begin{equation}\label{Edv}
E(p,q,r)=\sup_{\substack{\|x\cd\|_{\lpn}\le\delta\\\|\varphi_j\cd x\cd\|_{\lr}\le1,\ j=1,\ldots,n}}\|\psi\cd x\cd\|_{\lqq}.
\end{equation}

\section{The case of homogenous weight functions}
\label{hom}

Let $T$ be a cone in a linear space, $T_0=T$, $\mu\cd$ be a homogenous measure of degree $d$, $|\psi\cd|$ be homogenous function of degree $\eta$, $|\varphi_j\cd|$, $j=1,\ldots,n$, be homogenous functions of degrees $\nu$, $\psi(t)\ne0$ and $\sum_{j=1}^n|\varphi_j(t)|\ne0$ for almost all $t\in T$. Let assume, again, that $1\le p<q,r<\infty$. For $k\in[0,1)$ the function $k^{\frac1{p-q}}(1-k)^{-\frac1{r-q}}$ increases monotonically from $0$ to $+\infty$. Consequently, there exists $k\cd$ such that for almost all $t\in T$
\begin{equation}\label{kz}
\frac{k^{\frac1{p-q}}(t)}{(1-k(t))^{\frac1{r-q}}}=s_r^{-\frac1{r-q}}(t)|\psi(t)|
^{\frac{q(p-r)}{(p-q)(r-q)}},\quad s_r(t)=\sum_{j=1}^n|\varphi_j(t)|^r.
\end{equation}
Set
\begin{equation*}
k(t)=\begin{cases}\left(1-|\psi(t)|^{-q}s_q(t)\right)_+,&(p,q,r)\in P_1,\\\
\min\left\{1,|\psi(t)|^{-p}\right\}.&(p,q,r)\in P_2
\end{cases}
\end{equation*}

\begin{theorem}\label{S22}
Let $(p,q,r)\in P\cup P_1\cup P_2$ and $\nu+d(1/r-1/p)\ne0$.
Assume that for $(p,q,r)\in P\cup P_1$
\begin{align*}
I_1&=\iT|\psi(t)|^{\frac{qp}{p-q}}k^{\frac p{p-q}}(t)\,d\mu(t)<\infty,\\
I_{j+1}&=\iT|\psi(t)|^{\frac{qr}{p-q}}|\varphi_j(t)|^rk^{\frac r{p-q}}(z)\,d\mu(t)<\infty,\ j=1,\ldots,n,
\end{align*}
and for $(p,q,r)\in P_2$
\begin{align*}
I_1&=\iT\left(s_r^{-1}(t)(|\psi(t)|^p-1)_+\right)^{\frac p{r-p}}\,d\mu(t)<\infty,\\
I_{j+1}&=\iT|\varphi_j(t)|^r\left(s_r^{-1}(t)(|\psi(t)|^p-1)_+
\right)^{\frac r{r-p}}\,d\mu(t)<\infty,\ j=1,\ldots,n.
\end{align*}
Moreover, assume that $I_2=\ldots=I_{n+1}$. Then
\begin{equation}\label{EEpqr}
E(p,q,r)=\delta^\gamma
I_1^{-\gamma/p}I_2^{-(1-\gamma)/r}(I_1+nI_2)^{1/q},
\end{equation}
where
\begin{equation}\label{gam}
\gamma=\frac{\nu-\eta-d(1/q-1/r)}{\nu+d(1/r-1/p)}.
\end{equation}
The method
\begin{equation*}
\wm(y)(t)=k(\xi t)\psi(t)y(t),
\end{equation*}
where
\begin{equation}\label{yy}
\xi=\left(\delta I_1^{-1/p}I_2^{1/r}\right)^{\frac1{\nu+d(1/r-1/p)}},
\end{equation}
is optimal.
\end{theorem}

\begin{proof}
1. Let $(p,q,r)\in P$. Put
\begin{equation*}
\wx(t)=\left(\frac{q|\psi(t)|^q}{p\lambda_0}\right)^{\frac1{p-q}}k^{\frac1{p-q}}(\xi t),
\end{equation*}
where $\lambda_0$ will be specified later. We have
\begin{equation}\label{eq11}
p\lambda_0\wx^{p-q}(t)=q|\psi(t)|^qk(\xi t)
\end{equation}
and
\begin{equation*}
rc_r(t)\wx^{r-q}(t)=rc_r(t)\left(\frac{q|\psi(t)|^q}{p\lambda_0}\right)^{\frac{r-q}{p-q}}k^
{\frac{r-q}{p-q}}(\xi t).
\end{equation*}
Since $|\psi\cd|$ and $|\varphi_j\cd|$, $j=1,\ldots,n$, are homogenous it follows by \eqref{kz} that
\begin{equation*}
k^{\frac{r-q}{p-q}}(\xi t)=\frac{|\psi(\xi t)|^{\frac{q(p-r)}{p-q}}
}{c_r(\xi t)}(1-k(\xi t))
=\xi^{\eta\frac{q(p-r)}{p-q}-\nu r}\frac{|\psi(t)|^{\frac{q(p-r)}{p-q}}
}{c_r(t)}(1-k(\xi t)).
\end{equation*}
Thus,
\begin{equation*}
rc_r(t)\wx^{r-q}(t)=r\left(\frac q{p\lambda_0}\right)^{\frac{r-q}{p-q}}\xi^{\eta\frac{q(p-r)}{p-q}-\nu r}|\psi(t)|^q(1-k(\xi t)).
\end{equation*}
Put
\begin{equation}\label{l2}
\lambda=\frac qr\left(\frac q{p\lambda_0}\right)^{-\frac{r-q}{p-q}}\xi^{-\eta\frac{q(p-r)}{p-q}+\nu r}.
\end{equation}
Then
\begin{equation}\label{eq12}
r\lambda c_r(t)\wx^{r-q}(t)=q|\psi(t)|^q(1-k(\xi t)).
\end{equation}
Taking the sum of \eqref{eq11} and \eqref{eq12}, we obtain
\begin{equation*}
p\lambda_0\wx^{p-q}(t)+r\lambda c_r(t)\wx^{r-q}(t)=q|\psi(t)|^q.
\end{equation*}
It means that $\wx\cd$ satisfies \eqref{ko1} for $\lambda_1=\ldots=\lambda_n=\lambda$.

Now we show that for
\begin{equation}\label{l1}
\lambda_0=\frac qpI_1^{\frac{p-q}p}\xi^{-\eta q-d\frac{p-q}p}\delta^{q-p}
\end{equation}
the equalities
\begin{equation*}
\iT\wx^p(t)\,d\mu(t)=\delta^p,\quad\iT|\varphi_j(t)|^r\wx^r(t)\,d\mu(t)
=1,\ j=1,\ldots,n,
\end{equation*}
hold. In view of the definition of $\wx\cd$ we need to check that
\begin{align*}
\iT\left(\frac{q|\psi(t)|^q}{p\lambda_0}\right)^{\frac p{p-q}}k^{\frac p{p-q}}(\xi t)\,d\mu(t)&
=\delta^p,\\
\iT|\varphi_j(t)|^r\left(\frac{q|\psi(t)|^q}{p\lambda_0}\right)^{\frac r{p-q}}k^{\frac r{p-q}}(\xi t)\,d\mu(t)&
=1,\ j=1,\ldots,n.
\end{align*}
Changing $z=\xi t$ and taking into account that functions $|\psi\cd|$, $|\varphi_j\cd|$, $j=1,\ldots,n$, with the measure $\mu\cd$ are homogenous, we obtain
$$\left(\frac q{p\lambda_0}\right)^{\frac p{p-q}}I_1
=\delta^p\xi^{\frac{\eta qp}{p-q}+d},\quad
\left(\frac q{p\lambda_0}\right)^{\frac r{p-q}}I_{j+1}
=\xi^{\frac{\eta qr}{p-q}+\nu r+d},\ j=1,\ldots,n.$$
The validity of these equalities immediately follows from the definitions of $\lambda_0$ and $\xi$.

It follows by Theorem~\ref{T1}, \eqref{l1}, \eqref{l2}, and \eqref{yy} that
\begin{multline*}
E^q(p,q,r)=\frac{p\lambda_0\delta^p+nr\lambda}q=I_1^{\frac{p-q}p}
\xi^{-\eta q-d\frac{p-q}p}\delta^q\\+n\left(\frac{p\lambda_1}q\right)^{\frac{r-q}{p-q}}
\xi^{\nu r-\eta\frac{q(p-r)}{p-q}}
=\delta^{q\gamma}I_1^{-q\gamma/p}I_2^{-q(1-\gamma)/r}(I_1+nI_2).
\end{multline*}
Moreover, the same theorem states that the method
\begin{equation*}
\wm(y)(t)=q^{-1}p\lambda_0\wx^{p-q}(t)|\psi(t)|^{-q}\psi(t)y(t)=
k(\xi t)\psi(t)y(t)
\end{equation*}
is optimal.

2. Let $(p,q,r)\in P_1$. We use Theorem~\ref{T2}. Consider the function $\wx\cd$ defined by \eqref{xT2} with $\lambda_1=\ldots=\lambda_n=\lambda$. Let us find $\lambda_0$ and $\lambda$ from the conditions
\begin{equation*}
\iT\wx^p(t)\,d\mu(t)=\delta^p,\quad\iT|\varphi_j(t)|^q\wx^q(t)\,d\mu(t)=1,\ j=1,\ldots,n.
\end{equation*}
Then we obtain
\begin{gather*}
\left(\frac q{p\lambda_0}\right)^{\frac p{p-q}}\iT\left(|\psi(t)|^q-\lambda s_q(t)\right)_+^{\frac p{p-q}}\,d\mu(t)=\delta^p,\\
\left(\frac q{p\lambda_0}\right)^{\frac q{p-q}}\iT|\varphi_j(t)|^q\left(|\psi(t)|^q-\lambda s_q(t)\right)_+^{\frac q{p-q}}\,d\mu(t)=1,\ j=1,\ldots,n.
\end{gather*}
Put $\lambda=a^{(\eta-\nu)q}$, $a>0$. Changing $t=az$, we obtain
$$\left(\frac q{p\lambda_0}\right)^{\frac p{p-q}}a^{d+\frac{pq\eta}{p-q}}I_1=\delta^p,\quad
\left(\frac q{p\lambda_0}\right)^{\frac q{p-q}}a^{d+q\nu+\frac{q^2\eta}{p-q}}I_{j+1}=1,\ j=1,\ldots,n.$$
It is easy to check that these equalities are fulfilled for
\begin{equation*}
a=(I_1^{1/p}I_2^{-1/q}\delta^{-1})^{\frac1{\nu+d(1/q-1/p)}},\quad
\lambda_0=\frac qpI_1I_2^{-1}\delta^{-p}(I_1^{-q/p}I_2\delta^q)^{\frac{\eta-\nu}{\nu+d(1/q-1/p)}}.
\end{equation*}
Substituting these values in \eqref{Eq=r} and \eqref{met1} we obtain the statement of the theorem in the case under consideration.

3. Let $(p,q,r)\in P_2$. Here we use Theorem~\ref{TT3}. Put $\lambda_1=\ldots=\lambda_n=\lambda$
in the definition of $\wx\cd$ (see \eqref{xT3}). We find $\lambda_0$ and $\lambda$ from the conditions
\begin{equation*}
\iT\wx^p(t)\,d\mu(t)=\delta^p,\quad\iT|\varphi_j(t)|^r\wx^r(t)\,d\mu(t)=1,\ j=1,\ldots,n.
\end{equation*}
We have
\begin{gather*}
\left(\frac p{r\lambda}\right)^{\frac p{r-p}}\iT\left(s_r^{-1}(t)(|\psi(t)|^p-\lambda_0)_+\right)^{\frac p{r-p}}\,d\mu(t)=\delta^p,\\
\left(\frac p{r\lambda}\right)^{\frac r{r-p}}\iT|\varphi_j(t)|^r\left(s_r^{-1}(t)(|\psi(t)|^p-\lambda_0)_+\right)^{\frac r{r-p}}\,d\mu(t)=1,\ j=1,\ldots,n.
\end{gather*}
Put $\lambda_0=a^{\eta p}$, $a>0$. Changing $t=az$, we obtain
\begin{gather*}
\left(\frac p{r\lambda}\right)^{\frac p{r-p}}a^{d+\frac{p^2\eta}{r-p}-\frac{pr\nu}{r-p}}I_1=\delta^p,\\
\left(\frac p{r\lambda}\right)^{\frac r{r-p}}a^{d+r\nu+\frac{pr\eta}{r-p}-\frac{r^2\nu}{r-p}}I_{j+1}=1,\ j=1,\ldots,n.
\end{gather*}
These equalities are valid for
\begin{align*}
a&=(I_1^{1/p}I_2^{-1/r}\delta^{-1})^{\frac1{\nu+d(1/r-1/p)}},\\
\lambda&=\frac prI_1^{r/p-1}\delta^{p-r}(I_1^{r/p}I_2^{-1}\delta^{-r})^{\frac{p\eta /r-\nu-d(1/r-1/p)}{\nu+d(1/r-1/p)}}.
\end{align*}
It remains to substitute these values into \eqref{Ep=q} and \eqref{met11}.
\end{proof}

\begin{corollary}\label{Sdva}
Assume that conditions of Theorem~\ref{S22} hold. Then for all $x\cd\ne0$ such that $x\cd\in\lp$ and $\varphi_j\cd x\cd\in\lr$, $j=1,\ldots,n$, the sharp inequality
\begin{equation}\label{ncar}
\|\psi\cd x\cd\|_{\lqq}\le C\|x\cd\|_{\lp}^\gamma\left(\max_{1\le j\le n}\|\varphi_j\cd x\cd\|_{\lr}\right)^{1-\gamma}
\end{equation}
holds, where
\begin{equation*}
C=I_1^{-\gamma/p}I_2^{-(1-\gamma)/r}(I_1+nI_2)^{1/q}.
\end{equation*}
\end{corollary}

\begin{proof}
Let $x\cd\in\lp$, $\|\varphi_j\cd x\cd\|_{\lr}<\infty$, $j=1,\ldots,n$ and $x\cd\ne0$. Put
\begin{equation*}
A=\max_{1\le j\le n}\|\varphi_j\cd x\cd\|_{\lr}.
\end{equation*}
Consider $\wx\cd=x\cd/A$. Put $\delta=\|\wx\cd\|_{\lpn}$. Then $\|\varphi_j\cd \wx\cd\|_{\lr}\le1$, $j=1,\ldots,n$. In view of \eqref{Edv} and Theorem~\ref{S22} we have
\begin{equation*}
\|\psi\cd\wx\cd\|_{\lqq}\le C\|\wx\cd\|_{\lp}^\gamma.
\end{equation*}
This implies \eqref{ncar}.

If there exists a $\wCC<C$ for which \eqref{ncar} holds, then
\begin{equation*}
E(p,q,r)=\sup_{\substack{\|x\cd\|_{\lp}\le\delta\\\|\varphi_j\cd x\cd\|_{\lr}\le1,\ j=1,\ldots,n}}\|\psi\cd x\cd\|_{\lqq}\le\wCC\delta^\gamma<C\delta^\gamma.
\end{equation*}
This contradicts with \eqref{EEpqr}.
\end{proof}

Let $|w\cd|$, $|w_0\cd|$ be homogenous functions of degrees $\theta$, $\theta_0$, respectively and $|w_j\cd|$, $j=1,\ldots,n$, be homogenous functions of degree $\theta_1$. We assume that $w(t),w_0(t)\ne0$ and $\sum_{j=1}^n|w_j(t)|\ne0$ for almost all $t\in T$.

For $(p,q,r)\in P$ we define $\wk\cd$ by the equality
\begin{equation*}
\frac{\wk^{\frac1{p-q}}(t)}{(1-\wk(t))^{\frac1{r-q}}}=\left|\frac{w_0(t)}{w(t)}\right|
^{\frac p{p-q}}\biggl(\sum_{j=1}^n\left|\frac{w_j(t)}{w(t)}\right|^r\biggr)^{-\frac1{r-q}}.
\end{equation*}
For $(p,q,r)\in P_1$ set
\begin{equation*}
\wk(t)=\biggl(1-|w(t)|^{-q}\sum_{j=1}^n|w_j(t)|^q\biggr)_+.
\end{equation*}
Put
\begin{equation}\label{wg}
\wt=\theta+d/q,\quad\wt_0=\theta_0+d/p,\quad\wt_1=\theta_1+d/r,\quad
\wg=\frac{\wt_1-\wt}{\wt_1-\wt_0}.
\end{equation}

\begin{corollary}\label{Sl3}
Let $(p,q,r)\in P\cup P_1\cup P_2$ and $\wt_0\ne\wt_1$. Assume that for $(p,q,r)\in P\cup P_1$
\begin{align*}
\wII_1&=\iT\left|\frac{w(t)}{w_0(t)}\right|^{\frac{qp}{p-q}}\wk^{\frac p{p-q}}(t)\,d\mu(t)<\infty,\\
\wII_{j+1}&=\iT\frac{|w(t)|^{\frac{qr}{p-q}}}{|w_0(t)|^{\frac{pr}{p-q}}}|w_j(t)|^r\wk^{\frac r{p-q}}(t)\,d\mu(t)<\infty,\ j=1,\ldots,n,
\end{align*}
and for $(p,q,r)\in P_2$
\begin{align*}
\wII_1&=\iT|w_0(t)|^p\left(\frac{(|w(t)|^p-|w_0(t)|^p)_+}{\sum_{k=1}^n|w_k(t)|^r}\right)^{\frac p{r-p}}\,d\mu(t)<\infty,\\
\wII_{j+1}&=\iT|w_j(t)|^r\left(\frac{(|w(t)|^p-|w_0(t)|^p)_+}{\sum_{k=1}^n|w_k(t)|^r}\right)^{\frac r{r-p}}\,d\mu(t)<\infty,\ j=1,\ldots,n.
\end{align*}
Moreover, assume that $\wII_2=\ldots=\wII_{n+1}$. Then for all $x\cd\ne0$ such that $w_0\cd x\cd\in L_p(T,\mu)$ and $w_j\cd x\cd\in L_r(T,\mu)$, $j=1,\ldots,n$, the sharp inequality
\begin{equation}\label{CCC}
\|w\cd x\cd\|_{\lqq}\le\wCC\|w_0\cd x\cd\|_{\lp}^{\wg}\left(\max_{1\le j\le n}\|\omega_j\cd x\cd\|_{\lr}\right)^{1-\wg}
\end{equation}
holds, where
\begin{equation*}
\wCC=\wII_1^{-\wg/p}\wII_2^{-(1-\wg)/r}(\wII_1+n\wII_2)^{1/q}.
\end{equation*}
\end{corollary}

\begin{proof}
Set
\begin{equation*}
\psi(t)=\frac{w(t)}{w_0(t)},\quad\varphi_j(t)=\frac{w_j(t)}{w_0(t)},\ j=1,\ldots,n.
\end{equation*}
Then $|\psi\cd|$ is a homogenous function of degree $\eta=\theta-\theta_0$ and $|\varphi_j\cd|$, $j=1,\ldots,n$, are homogenous functions of degrees $\nu=\theta_1-\theta_0$. The quantity $\gamma$ which was defined by \eqref{gam} has the following form:
\begin{equation*}
\wg=\frac{\wt_1-\wt}{\wt_1-\wt_0}.
\end{equation*}
It follows by Corollary~\ref{Sdva} that for all $y\cd\ne0$ such that $y\cd\in L_p(T,\mu)$ and $\varphi_j\cd y\cd\in L_r(T,\mu)$, $j=1,\ldots,n$, the sharp inequality
\begin{equation*}
\|\psi\cd y\cd\|_{\lqq}\le\wCC\|y\cd\|_{\lp}^{\wg}\left(\max_{1\le j\le n}\|\varphi_j\cd y\cd\|_{\lr}\right)^{1-\wg}.
\end{equation*}
holds.
Substituting $y\cd=w_0\cd x\cd$, we obtain \eqref{CCC}.
\end{proof}

\section{Homogenous weights in $\mathbb R^d$}\label{rd}

Let $T$ be a cone in $\mathbb R^d$, $d\mu(t)=dt$, $|\psi\cd|$ be homogenous function of degree $\eta$, $|\varphi_j\cd|$, $j=1,\ldots,n$, be homogenous functions of degrees $\nu$, $\psi(t)\ne0$ and $\sum_{j=1}^n|\varphi_j(t)|\ne0$ for almost all $t\in T$. Consider the polar transformation
\begin{equation*}
\arraycolsep=0.08em
\begin{array}{rcl}
t_1&=&\rho\cos\omega_1,\\
t_2&=&\rho\sin\omega_1\cos\omega_2,\\
\hdotsfor{3}\\
t_{d-1}&=&\rho\sin\omega_1\sin\omega_2\ldots\sin\omega_{d-2}\cos\omega_{d-1},\\
t_d&=&\rho\sin\omega_1\sin\omega_2\ldots\sin\omega_{d-2}\sin\omega_{d-1}.
\end{array}
\end{equation*}
Set $\omega=(\omega_1,\ldots,\omega_{d-1})$. For any function $f\cd$ we put
\begin{equation}\label{wtw}
\widetilde f(\omega)=|f(\cos\omega_1,\ldots,\sin\omega_1\sin\omega_2\ldots\sin\omega_{d-2}\sin\omega_{d-1})|.
\end{equation}
Note that if $|f\cd|$ is a homogenous function of degree $\kappa$, then $f(\omega)=\rho^{-\kappa}|f(t)|$. Denote by $\Omega$ the range of $\omega$. Since $T$ is a cone, $\Omega$ does not depend on $\rho$. Put
\begin{equation*}
J(\omega)=\sin^{d-2}\omega_1\sin^{d-3}\omega_2\ldots\sin\omega_{d-2}.
\end{equation*}

Assume that $\gamma\in(0,1)$, where $\gamma$ is defined by \eqref{gam}. Put
\begin{equation}\label{qu}
\frac1{q^*}=\frac1q-\frac\gamma p-\frac{1-\gamma}r.
\end{equation}
It is easy to verify that $q^*>q\ge1$. Moreover,
$$q^*=\frac{pqr(\nu+d(1/r-1/p))}{\nu r(p-q)-\eta q(p-r)}.$$

\begin{theorem}\label{T3}
Let $(p,q,r)\in P\cup P_1\cup P_2$ and $\gamma\in(0,1)$. Assume that
\begin{equation*}
I=\int_\Omega\frac{\wps^{q^*}(\omega)}{\wc_r^{\,q^*(1-\gamma)/r}(\omega)
}J(\omega)\,d\omega<\infty,
\end{equation*}
and $I'_1=\ldots=I'_n$, where
\begin{equation*}
I'_j=\int_\Omega\frac{\wps^{q^*}(\omega)\wva_j^r(\omega)}
{\wc_r^{\,q^*(1-\gamma)/r+1}(\omega)}J(\omega)\,d\omega,\ j=1,\ldots,n.
\end{equation*}
Then
\begin{equation}\label{kpr}
E(p,q,r)=K\delta^\gamma,
\end{equation}
where
\begin{equation*}
K=\gamma^{-\frac\gamma p}\left(\frac{1-\gamma}n\right)^{-\frac{1-\gamma}r}\Biggl(\frac{B\left(q^*\gamma /p,q^*(1-\gamma)/r\right)I}{|\nu+d(1/r-1/p)|(\gamma r+(1-\gamma)p)}\Biggr)^{1/q^*},
\end{equation*}
where $B(\cdot,\cdot)$ is the Eiler beta-function. Moreover, the method
\begin{equation*}
\wm(y)(t)=\kappa\left(\wxi^{\frac1{\nu+d(1/r-1/p)}}t\right)\psi(t)y(t),
\end{equation*}
where
\begin{equation*}
\wxi=\delta\gamma^{-1/p}\left(\frac{1-\gamma}n\right)^{1/r}\Biggl(\frac{B\left(q^*\gamma /p,q^*(1-\gamma)/r\right)I}{|\nu+d(1/r-1/p)|(\gamma r+(1-\gamma)p)}\Biggr)^{1/r-1/p},
\end{equation*}
is optimal recovery method.
\end{theorem}

\begin{proof}
First of all, we note that $I_1'+\ldots+I_n'=I$. Consequently, $I_j'=I/n$, $j=1,\ldots,n$. We will apply Theorem~\ref{S22}.

1. Let $(p,q,r)\in P$. Passing to the polar transformation we obtain
\begin{equation*}
\frac{k^{\frac1{p-q}}(\rho,\omega)}{(1-k(\rho,\omega))^{\frac1{r-q}}}=\rho^{\frac{\eta q(p-r)-\nu r(p-q)}{(p-q)(r-q)}}\frac{\wps^{\frac{q(p-r)}{(p-q)(r-q)}}(\omega)}
{\wc_r^{\frac1{r-q}}(\omega)}.
\end{equation*}
Using the same scheme of calculation of $I_1$ as it was given in \cite[Theorem~3]{Os}, we obtain
\begin{equation*}
I_1=\frac\gamma{pr|\nu+d(1/r-1/p)|}\left(\frac\gamma p+\frac{1-\gamma}r\right)^{-1}B(\wpp,\wqq)I,
\end{equation*}
where
\begin{equation*}
\wpp=q^*\frac{\gamma}p,\quad\wqq=q^*\frac{1-\gamma}r.
\end{equation*}
In a similar way we calculate
\begin{equation*}
I_{j+1}=\frac{1-\gamma}{pr|\nu+d(1/r-1/p)|}\left(\frac\gamma p+\frac{1-\gamma}r\right)^{-1}B(\wpp,\wqq)I_j',\ j=1,\ldots,n.
\end{equation*}
Thus,
\begin{equation*}
I_2=\frac{1-\gamma}{npr|\nu+d(1/r-1/p)|}\left(\frac\gamma p+\frac{1-\gamma}r\right)^{-1}B(\wpp,\wqq)I.
\end{equation*}
It remains to substitute these values into \eqref{EEpqr} and \eqref{yy}.

2. Let $(p,q,r)\in P_1$. Now we use the scheme of calculation of $I_1$ which was given in \cite[Theorem~3]{Os22}. We obtain
\begin{multline*}
I_1=\frac I{|\nu-\eta|q}B\left(q^*\gamma/p+2,q^*(1-\gamma)/q\right)\\
=\frac I{|\nu-\eta|q}\frac{q^*\gamma/p+1}{q^*\gamma/p+1+q^*(1-\gamma)/q}
B\left(q^*\gamma/p+1,q^*(1-\gamma)/q\right).
\end{multline*}
Since $r=q$ we have
\begin{equation*}
\frac1{q^*}=\gamma\left(\frac1q-\frac1p\right),\quad\gamma=\frac{\nu-\eta}{\nu+d(1/q-1/p)}.
\end{equation*}
Therefore, $q^*\gamma/p+1=q^*\gamma/q$. Hence
\begin{multline*}
I_1=\frac{I\gamma}{|\nu-\eta|q}B\left(q^*\gamma/p+1,q^*(1-\gamma)/q\right)\\
=\frac{I\gamma}{|\nu-\eta|q}\frac{q^*\gamma/p}{q^*\gamma/p+q^*(1-\gamma)/q}
B\left(q^*\gamma/p,q^*(1-\gamma)/q\right)\\
=\frac\gamma{pr|\nu+d(1/r-1/p)|}\left(\frac\gamma p+\frac{1-\gamma}r\right)^{-1}B(\wpp,\wqq)I.
\end{multline*}
By the similar way we get
\begin{multline*}
I_{j+1}=\frac{I_j'}{|\nu-\eta|q}B\left(q^*\gamma/p+1,q^*(1-\gamma)/q+1\right)\\
=\frac{I_j'}{|\nu-\eta|q}\frac{q^*\gamma/p}{q^*\gamma/p+q^*(1-\gamma)/q+1}
B\left(q^*\gamma/p,q^*(1-\gamma)/q+1\right)\\
=\frac{I_j'\gamma}{|\nu-\eta|p}B\left(q^*\gamma/p,q^*(1-\gamma)/q+1\right)
=\frac{I_j'\gamma}{|\nu-\eta|p}\frac{q^*(1-\gamma)/q}{q^*\gamma/p+q^*(1-\gamma)/q}
B\left(q^*\gamma/p,q^*(1-\gamma)/q\right)\\
=\frac{1-\gamma}{npr|\nu+d(1/r-1/p)|}\left(\frac\gamma p+\frac{1-\gamma}r\right)^{-1}B(\wpp,\wqq)I.
\end{multline*}
Thus, we obtain the same formulas for $I_1$ and $I_2$ as in the first case.

3. Let $(p,q,r)\in P_2$. Here we use the scheme of calculation of $J_1$ and $J_2$ which was given in \cite[Theorem~3]{Os22}. We obtain
\begin{align*}
I_1&=\frac I{|\eta|p}B\left(q^*\gamma/p+1,q^*(1-\gamma)/r+1\right),\\
I_{j+1}&=\frac{I_j'}{|\eta|p}B\left(q^*\gamma/p,q^*(1-\gamma)/r+2\right),\ j=1,\ldots,n.
\end{align*}
Since $q=p$ we have
\begin{equation*}
\frac1{q^*}=(1-\gamma)\left(\frac1p-\frac1r\right),\quad
1-\gamma=\frac\eta{\nu+d(1/r-1/p)}.
\end{equation*}
Therefore, $q^*(1-\gamma)/r+1=q^*(1-\gamma)/p$. Hence
\begin{multline*}
I_1=\frac I{|\eta|p}\frac{q^*\gamma/p}{q^*\gamma/p+q^*(1-\gamma)/r+1}
B\left(q^*\gamma/p,q^*(1-\gamma)/r+1\right)\\
=\frac{I\gamma}{|\eta|p}B\left(q^*\gamma/p,q^*(1-\gamma)/r+1\right)
=\frac{I\gamma}{|\eta|p}
\frac{q^*(1-\gamma)/r}{q^*\gamma/p+q^*(1-\gamma)/r}B\left(q^*\gamma/p,q^*(1-\gamma)/r\right)\\
=\frac\gamma{pr|\nu+d(1/r-1/p)|}\left(\frac\gamma p+\frac{1-\gamma}r\right)^{-1}B(\wpp,\wqq)I.
\end{multline*}
For $I_{j+1}$, $j=1,\ldots,n$, we have
\begin{multline*}
I_{j+1}=\frac{I_j'}{|\eta|p}\frac{q^*(1-\gamma)/r+1}{q^*\gamma/p+q^*(1-\gamma)/r+1}
B\left(q^*\gamma/p,q^*(1-\gamma)/r+1\right)\\
=\frac{I_j'(1-\gamma)}{|\eta|p}B\left(q^*\gamma/p,q^*(1-\gamma)/r+1\right)
=\frac{I_j'(1-\gamma)}{pr|\nu+d(1/r-1/p)|}\left(\frac\gamma p+\frac{1-\gamma}r\right)^{-1}B(\wpp,\wqq)\\
=\frac{1-\gamma}{npr|\nu+d(1/r-1/p)|}\left(\frac\gamma p+\frac{1-\gamma}r\right)^{-1}B(\wpp,\wqq)I.
\end{multline*}
Again we obtain the same formulas for $I_1$ and $I_2$ as in the previous cases.
\end{proof}

For $n=1$ Theorem~\ref{T3} was proved in \cite{Os22}. Analogously to Corollary~\ref{Sdva} we obtain

\begin{corollary}\label{Sdva1}
Assume that conditions of Theorem~\ref{T3} hold. Then for all $x\cd$ such that $x\cd\in\lp$ and $\varphi_j\cd x\cd\in\lr$, $j=1,\ldots,n$, the sharp inequality
\begin{equation*}
\|\psi\cd x\cd\|_{\lqq}\le K\|x\cd\|_{\lp}^\gamma\left(\max_{1\le j\le n}\|\varphi_j\cd x\cd\|_{\lr}\right)^{1-\gamma}
\end{equation*}
holds.
\end{corollary}

Let $|w\cd|$, $|w_0\cd|$ be homogenous functions of degrees $\theta$, $\theta_0$, respectively and $|w_j\cd|$, $j=1,\ldots,n$, be homogenous functions of degree $\theta_1$. We assume that $w(t),w_0(t)\ne0$ and $\sum_{j=1}^n|w_j(t)|\ne0$ for almost all $t\in T$.
Define $\wtw\cd$, $\wtw_0\cd$, $\wtw_1\cd$ by \eqref{wtw}. Similar to Corollary~\ref{Sl3} we obtain

\begin{corollary}\label{corT3}
Let $(p,q,r)\in P\cup P_1\cup P_2$ and $\wg\in(0,1)$ where $\wg$ is defined by \eqref{wg}. Assume that
\begin{equation*}
\wII=\int_\Omega\frac{\wtw^{\vq}(\omega)}{\wtw_0^{\vq\wg}(\omega)\left(\sum_{k=1}^n
\wtw_k^r(\omega)\right)^{\vq(1-\wg)/r}}J(\omega)\,d\omega<\infty,
\end{equation*}
where
\begin{equation*}
\frac1{\vq}=\frac1q-\frac{\wg}p-\frac{1-\wg}r,
\end{equation*}
and $\wII'_1=\ldots=\wII'_n$, where
\begin{equation*}
\wII'_j=\int_\Omega\frac{\wtw^{\vq}(\omega)\wtw^r_j(\omega)}{\wtw_0^{\vq\wg}(\omega)\left(\sum_{k=1}^n
\wtw_k^r(\omega)\right)^{\vq(1-\wg)/r+1}}J(\omega)\,d\omega,\ j=1,\ldots,n.
\end{equation*}
Then for all $x\cd$ such that $w_0\cd x\cd\in L_p(T,\mu)$ and $w_j\cd x\cd\in L_r(T,\mu)$, $j=1,\ldots,n$, the sharp inequality
\begin{equation*}
\|w\cd x\cd\|_{\lqq}\le\wC\|w_0\cd x\cd\|_{\lp)}^{\wg}\left(\max_{1\le j\le n}\|\omega_j\cd x\cd\|_{\lr}\right)^{1-\wg}
\end{equation*}
holds, where
\begin{equation}\label{wC}
\wC=\wg^{-\frac{\wg}p}\left(\frac{1-\wg}n\right)^{-\frac{1-\wg}r}\Biggl(\frac{B\left(\vq\wg/p,\vq
(1-\wg)/r\right)\wII}{|\wt_1-\wt_0|(\wg r+(1-\wg)p)}\Biggr)^{1/\vq}.
\end{equation}
\end{corollary}

The statement of Corollary~\ref{corT3} for $(p,q,r)\in P$ and $n=1$ was proved in \cite{B}.

We give an example of weights for which conditions of Corollary~\ref{corT3} hold. Let $T=\mathbb R^d_+$, $\theta_1>0$,
\begin{equation}\label{wei}
w(t)=(t_1^2+\ldots+t_d^2)^{\theta/2},\quad w_0(t)=(t_1^2+\ldots+t_d^2)^{\theta_0/2},\quad w_j(t)=t_j^{\theta_1},\ j=1,\ldots,d.
\end{equation}
The condition $0<\wg<1$ is equivalent to inequalities $\wt_1>\wt>\wt_0$ or $\wt_1<\wt<\wt_0$. Therefore, we assume that for $\theta$ and $\theta_0$ inequalities $\theta_1+d(1/r-1/q)>\theta>\theta_0+d(1/p-1/q)$ or $\theta_1+d(1/r-1/q)<\theta<\theta_0+d(1/p-1/q)$ hold.

It is easy to check that $\wtw\cd=\wtw_0\cd=1$ and $\wtw_j(\omega)={\wtt_j\hspace{-2pt}}^{\theta_1}(\omega)$, $j=1,\ldots,d$, where
\begin{equation*}
\arraycolsep=0.08em
\begin{array}{rcl}
\wtt_1(\omega)&=&\cos\omega_1,\\
\wtt_2(\omega)&=&\sin\omega_1\cos\omega_2,\\
\hdotsfor{3}\\
\wtt_{d-1}(\omega)&=&\sin\omega_1\sin\omega_2\ldots\sin\omega_{d-2}\cos\omega_{d-1},\\
\wtt_d(\omega)&=&\sin\omega_1\sin\omega_2\ldots\sin\omega_{d-2}\sin\omega_{d-1}.
\end{array}
\end{equation*}
Note that
\begin{equation*}
\sum_{k=1}^dt_k^2(\omega)=1.
\end{equation*}

For $\wII$ we have
\begin{equation}\label{wII}
\wII=\int_{\Pi_+^{d-1}}\frac{J(\omega)\,d\omega}{\left(\sum_{k=1}^d
{\wtt_k\hspace{-3pt}}^{r\theta_1}(\omega)\right)^{\vq(1-\wg)/r}},\quad\Pi_+^{d-1}=[0,\pi/2]^{d-1}.
\end{equation}
If $r\theta_1\le2$, then
\begin{equation}\label{etd1}
\sum_{k=1}^d
{\wtt_k\hspace{-3pt}}^{r\theta_1}(\omega)\ge\sum_{k=1}^d
{\wtt_k\hspace{-3pt}}^2(\omega)=1.
\end{equation}
For $r\theta_1>2$ by H\"older's inequality
\begin{equation*}
1=\sum_{k=1}^d
{\wtt_k\hspace{-3pt}}^2(\omega)\le\biggl(\sum_{k=1}^d
{\wtt_k\hspace{-3pt}}^{r\theta_1}(\omega)\biggr)^{\frac2{r\theta_1}}d^{1-\frac2{r\theta_1}}.
\end{equation*}
Thus,
\begin{equation}\label{etd2}
\sum_{k=1}^d
{\wtt_k\hspace{-3pt}}^{r\theta_1}(\omega)\ge d^{1-\frac{r\theta_1}2}.
\end{equation}
It follows by \eqref{etd1} and \eqref{etd2} that $\wII<\infty$.

For $\wII'_j$ we have
\begin{equation*}
\wII'_j=\int_{\Pi_+^{d-1}}\frac{\wtt_j\hspace{-3pt}^{r\theta_1}J(\omega)\,d\omega}
{\left(\sum_{k=1}^d
{\wtt_k\hspace{-3pt}}^{r\theta_1}(\omega)\right)^{\vq(1-\wg)/r+1}},\ j=1,\ldots,d.
\end{equation*}
Consider the integrals
\begin{equation*}
L_j=\int_{\mathbb R_+^d\cap\mathbb B^d}\frac{\left(\sum_{k=1}^dt_k^2\right)^{\theta_1\vq(1-\wg)/2}t_j^{r\theta_1}}
{\left(\sum_{k=1}^dt_k^{r\theta_1}\right)^{\vq(1-\wg)/r+1}}\,dt,\ j=1,\ldots,d,
\end{equation*}
where $\mathbb B^d$ is the unit ball in $\mathbb R^d$. If we change variables in $L_j$ changing places variables $t_j$ and $t_k$, then $L_j$ passes to $L_k$. Therefore, $L_1=\ldots=L_d$. Passing to the polar transformation we obtain that $L_j=\wII'_j/d$, $j=1,\ldots,d$. Consequently, $\wII'_1=\ldots=\wII'_d$.

Thus, we obtain

\begin{corollary}
Let $(p,q,r)\in P\cup P_1\cup P_2$, $\theta_1>0$, $\theta$ and $\theta_0$ be such that $\theta_1+d(1/r-1/q)>\theta>\theta_0+d(1/p-1/q)$ or $\theta_1+d(1/r-1/q)<\theta<\theta_0+d(1/p-1/q)$. Then for weights \eqref{wei} and all $x\cd$ for which $w_0\cd x\cd\in L_p(\mathbb R_+^d)$ and $w_j\cd x\cd\in L_r(\mathbb R_+^d)$, $j=1,\ldots,d$, the sharp inequality
\begin{equation*}
\|w\cd x\cd\|_{L_q(\mathbb R_+^d)}\le\wC\|w_0\cd x\cd\|_{L_p(\mathbb R_+^d)}^{\wg}\left(\max_{1\le j\le d}\|\omega_j\cd x\cd\|_{L_r(\mathbb R_+^d)}\right)^{1-\wg}
\end{equation*}
holds, where $\wC$ is defined by \eqref{wC} in which the value $\wII$ is defined by \eqref{wII}.
\end{corollary}

We give one more example.

\begin{corollary}\label{cor7}
Let $(p,q,r)\in P\cup P_1\cup P_2$, weights $w\cd$, $w_0\cd$, $w_1\cd$ be defined by \eqref{wei} for $\theta=d(1-1/q)$, $\theta_0=d-(\lambda+d)/p$, $\theta_1=d+(\mu-d)/r$, where $\lambda,\mu>0$. Put
\begin{equation*}
\alpha=\frac\mu{p\mu+r\lambda},\quad\beta=\frac\lambda{p\mu+r\lambda}.
\end{equation*}
Then for all $x\cd$ such that $w_0\cd x\cd\in L_p(\mathbb R_+^d)$ and $w_j\cd x\cd\in L_r(\mathbb R_+^d)$, $j=1,\ldots,d$, the sharp inequality
$$\|w\cd x\cd\|_{L_q(\mathbb R_+^d)}\le C\|w_0\cd x\cd\|_{L_p(\mathbb R_+^d)}^{p\alpha}\left(\max_{1\le j\le d}\|\omega_j\cd x\cd\|_{L_r(\mathbb R_+^d)}\right)^{r\beta}$$
holds, where
\begin{equation*}
C=\frac{d^\beta}{(p\alpha)^\alpha(r\beta)^\beta}\left(\frac I{\lambda+\mu}
B\left(\frac\alpha{1/q-\alpha-\beta},\frac\beta{1/q-\alpha-\beta}\right)
\right)^{1/q-\alpha-\beta},
\end{equation*}
and
\begin{equation*}
I=\int_{\Pi_+^{d-1}}\frac{J(\omega)\,d\omega}{\left(\sum_{k=1}^d
{\wtt_k\hspace{-3pt}}^{r(d-1)+\mu}(\omega)\right)^{\frac\beta{1/q-\alpha-\beta}}}.
\end{equation*}
\end{corollary}

For $d=1$, $q=1$, and $(p,1,r)\in P$ the statement of Corollary~\ref{cor7} was proved in \cite{Le}.

\section{Recovery of differential operators from a noisy Fourier transform}\label{dif}

Let $T$ be a cone in $\mathbb R^d$, $d\mu(t)=dt$, $|\psi\cd|$ be homogenous function of degree $\eta$, $|\varphi_j\cd|$, $j=1,\ldots,n$, be homogenous functions of degrees $\nu$, $\psi(t)\ne0$ and $\sum_{j=1}^n|\varphi_j(t)|\ne0$ for almost all $t\in T$.

Let $S$ be the Schwartz space of rapidly decreasing $C^\infty$-functions on  $\mathbb R^d$, $S'$ be the corresponding space of distributions, and let
$F\colon S'\to S'$ be the Fourier transform. Set
\begin{equation*}
X_p=\left\{\,x\cd\in S':\varphi_j\cd Fx\cd\in\ld,\ j=1,\ldots,n,\ Fx\cd\in\lpd\,\right\}.
\end{equation*}
We define operators $D_j$, $j=1,\ldots,n$, as follows
$$D_jx\cd=F^{-1}(\varphi_j\cd Fx\cd)\cd,\ j=1,\ldots,n.$$
Put
\begin{equation}\label{Lam}
\Lambda x\cd=F^{-1}(\psi\cd Fx\cd)\cd.
\end{equation}

Consider the problem of the optimal recovery of values of the operator $\Lambda$ on the
class
\begin{equation*}
W_p^{\mathcal D}=\left\{\,x\cd\in X_p:\|D_jx\cd\|_{\ld}\le1,\ j=1,\ldots,n\,\right\},\quad\mathcal D=(D_1,\ldots,D_n),
\end{equation*}
from the noisy Fourier transform of the function $x\cd$. We assume that for each $x\cd\in W_p$
one knows a function $y\cd\in\lpd$ such that $\|Fx\cd-y\cd\|_{\lpd}\le\delta$, $\delta>0$. It is required to recover the function $\Lambda x\cd$ from $y\cd$. Assume that $\Lambda x\cd\in L_q(\mathbb R^d)$ for all $x\cd\in X_p$. As recovery methods we consider all possible mappings $m\colon\lpd\to L_q(\mathbb R^d)$. The error of a method $m$ is defined by
\begin{equation*}
e_{pq}(\Lambda,\mathcal D,m)=\sup_{\substack{x\cd\in W_p^{\mathcal D},\ y\cd\in\lpd\\\|Fx\cd-y\cd\|_{\lpd}\le\delta}}\|\Lambda x\cd-m(y)\cd\|_{L_q(\mathbb R^d)}.
\end{equation*}
The quantity
\begin{equation}\label{ED}
E_{pq}(\Lambda,\mathcal D)=\inf_{m\colon\lpd\to\ld}e_{pq}(\Lambda,\mathcal D,m)
\end{equation}
is called the error of optimal recovery, and the method on which the infimum is
attained, an optimal method.

\subsection{Recovery in the metric $\ld$}

By Plancherel's theorem,
\begin{equation*}
\|\Lambda x\cd-m(y)\cd\|_{\ld}=\frac1{(2\pi)^{d/2}}\|\wL x\cd-F(m(y))\cd\|_{\ld},
\end{equation*}
where $\wL x\cd=\psi\cd Fx\cd$. Moreover,
\begin{equation*}
\|D_jx\cd\|_{\ld}=\frac1{(2\pi)^{d/2}}\|\varphi_j\cd Fx\cd\|_{\ld},\ j=1,\ldots,n.
\end{equation*}
So, the problem under consideration coincides, up to a factor of $(2\pi)^{-d/2}$, with
problem \eqref{Ep} for $q=r=2$ with $\varphi_j\cd$ replaced by $(2\pi)^{-d/2}\varphi_j\cd$, $j=1,\ldots,n$.

For $q=r=2$ we denote by $\wwg$ and $\wwq$ the values $\gamma$ and $q^*$, which where defined by \eqref{gam} and \eqref{qu}:
\begin{equation*}
\wwg=\frac{\nu-\eta}{\nu+d(1/2-1/p)},\quad\wwq=\frac1{\wwg(1/2-1/p)}.
\end{equation*}
Set
\begin{equation*}
C_p(\nu,\eta)=\wwg^{-\frac\wwg p}
\left(\frac{1-\wwg}n\right)^{-\frac{1-\wwg}2}\Biggl(\frac{B\left(\wwq\wwg/p+1,
\wwq(1-\wwg)/2\right)}{2|\nu-\eta|}\Biggr)^{1/\wwq}.
\end{equation*}

\begin{theorem}\label{pred1}
Let $2<p\le\infty$, $\wwg\in(0,1)$. Assume that
\begin{equation}\label{Ieq}
I=\int_{\Pi^{d-1}}\frac{\wps^{\wwq}(\omega)}{\wc_2^{\,\wwq(1-\wwg)/2}(\omega)}
J(\omega)\,d\omega<\infty,\quad\Pi^{d-1}=[0,\pi]^{d-2}\times[0,2\pi]
\end{equation}
and $I'_1=\ldots=I'_n$, where
\begin{equation}\label{Ieq1}
I'_j=\int_{\Pi^{d-1}}\frac{\wps^{\wwq}(\omega)\wva_j^2(\omega)}
{\wc_2^{\,\wwq(1-\wwg)/2+1}(\omega)}J(\omega)\,d\omega,\ j=1,\ldots,n.
\end{equation}
Then
\begin{equation}\label{kpr1}
E_{p2}(\Lambda,\mathcal D)=\frac1{(2\pi)^{d\wwg/2}}C_p(\nu,\eta)I^{1/\wwq}\delta^{\wwg}.
\end{equation}
The method
\begin{equation}\label{mbe}
\wm(y)(t)=F^{-1}\left(\left(1-\beta\frac{s_2(t)}{|\psi(t)|^2}\right)_+
\psi(t)y(t)\right),
\end{equation}
where
\begin{equation*}
\beta=\frac{1-\wwg}{n(2\pi)^{d\wwg}}C_p^2(\nu,\eta)\left(\delta I^{1/2-1/p}\right)^{2\wwg},
\end{equation*}
is optimal.

Moreover, the sharp inequality
\begin{multline}\label{ene}
\|\Lambda x\cd\|_{\ld}\\
\le\frac1{(2\pi)^{d\wwg/2}}C_p(\nu,\eta)I^{1/\wwq}\|Fx\cd\|_{\lpd}^{\wwg}
\left(\max_{1\le j\le n}\|D_jx\cd\|_{\ld}\right)^{1-\wwg}
\end{multline}
holds.
\end{theorem}

\begin{proof}
Let $2<p<\infty$. By Theorem~\ref{T3} we have
\begin{equation*}
E_{p2}(\Lambda,\mathcal D)=\frac1{(2\pi)^{d\wwg/2}}K\delta^{\wwg},
\end{equation*}
where
\begin{equation*}
K=\wwg^{-\frac\wwg p}\left(\frac{1-\wwg}n\right)^{-\frac{1-\wwg}2}\Biggl(\frac{B\left(\wwq\wwg /p,\wwq(1-\wwg)/2\right)I}{|\nu+d(1/2-1/p)|(2\wwg+(1-\wwg)p)}\Biggr)^{1/\wwq}.
\end{equation*}
From the properties of the beta-function we find that
\begin{multline}\label{bf}
\frac{B\left(\wwq\wwg /p,\wwq(1-\wwg)/2\right)}{|\nu+d(1/2-1/p)|(2\wwg+(1-\wwg)p)}\\=
\frac{B\left(\wwq\wwg /p+1,\wwq(1-\wwg)/2\right)(\wwq\wwg /p+\wwq(1-\wwg)/2)}{|\nu+d(1/2-1/p)|(2\wwg+(1-\wwg)p)\wwq\wwg /p}\\
=\frac{B\left(\wwq\wwg /p+1,\wwq(1-\wwg)/2\right)}{2|\nu-\eta|}.
\end{multline}
Thus, equality \eqref{kpr1} holds.

It follows by Theorem~\ref{T3} that the method
\begin{equation*}
\wm(y)(t)=\left(1-\frac{\wxi^{2\wwg}c_2(t)}{(2\pi)^d|\psi(t)|^2}\right)_+\psi(t)y(t),
\end{equation*}
where
\begin{equation*}
\wxi=\delta(2\pi)^{d\frac{1-\wwg}{2\wwg}}\wwg^{-1/p}\left(\frac{1-\wwg}n\right)^{1/2}
\Biggl(\frac{B\left(\wwq\wwg /p,\wwq(1-\wwg)/2\right)I}{|\nu+d(1/2-1/p)|(2\wwg+(1-\wwg)p)}\Biggr)^{1/2-1/p},
\end{equation*}
is optimal. In view of \eqref{bf} we obtain
\begin{multline*}
\frac{\wxi^{2\wwg}}{(2\pi)^d}=\frac{\delta^{2\wwg}
\wwg^{-2\wwg/p}}{(2\pi)^{d\wwg}}\left(\frac{1-\wwg}n\right)^{\wwg}\Biggl(\frac{B\left(\wwq\wwg /p+1,\wwq(1-\wwg)/2\right)I}{2|\nu-\eta|}\Biggr)^{2\wwg(1/2-1/p)}\\
=\frac{1-\wwg}{n(2\pi)^{d\wwg}}C_p^2(\nu,\eta)\left(\delta I^{1/2-1/p}\right)^{2\wwg}.
\end{multline*}

Inequality \eqref{ene} follows from Corollary~\ref{Sdva1}. Consider the case $p=\infty$. It follows by Lemma~\ref{L1} that
\begin{equation}\label{EDD}
E_{\infty2}(\Lambda,\mathcal D)\ge\sup_{\substack{x\cd\in W_\infty^{\mathcal D}\\\|Fx\cd\|_{L_\infty(\mathbb R^d)}\le\delta}}\|\Lambda x\cd\|_{L_2(\mathbb R^d)}.
\end{equation}
Let $\wx\cd$ be such that
\begin{equation*}
F\wx(\xi)=\begin{cases}
\delta,&|\psi(\xi)|>\lambda\sqrt{s_2(\xi)},\\
0,&|\psi(\xi)|\le\lambda\sqrt{s_2(\xi)}.
\end{cases}
\end{equation*}

We show that $\lambda>0$ may be selected from the condition
\begin{equation*}
\frac1{(2\pi)^d}\iRd|\varphi_j(\xi)|^2|F\wx(\xi)|^2\,d\xi=1,\ j=1,\ldots,n.
\end{equation*}
Thus, $\lambda>0$ should be chosen from the condition
\begin{equation*}
\delta^2\int_{|\psi(\xi)|>\lambda\sqrt{s_2(\xi)}}|\varphi_j
(\xi)|^2\,d\xi=(2\pi)^d.
\end{equation*}
Passing to the polar transformation for $\nu>\eta$ we obtain
$$\delta^2\int_{\Pi_{d-1}}\wva_j^2(\omega)J(\omega)\,d\omega\int_0^{\Phi_1(\omega)}
\rho^{2\nu+d-1}\,d\rho=(2\pi)^d,\quad
\Phi_1(\omega)=\left(\frac{\wps(\omega)}{\lambda\sqrt{\wc_2(\xi)}}\right)^{\frac1{\nu-\eta}}.$$
If $\nu<\eta$, then $2\nu+d<0$ (since $\wwg\in(0,1)$) and we have
\begin{equation*}
\delta^2\int_{\Pi_{d-1}}\wva_j^2(\omega)J(\omega)\,d\omega\int_{\Phi_1(\omega)}^{+\infty}
\rho^{2\nu+d-1}\,d\rho=(2\pi)^d.
\end{equation*}
Hence
\begin{equation*}
\frac{\delta^2}{|2\nu+d|}\lambda^{-\frac{2\nu+d}{\nu-\eta}}I_j'=(2\pi)^d.
\end{equation*}
As already noted, it follows from the equality $I_1'+\ldots+I_n'=I$ that $I_j'=I/n$, $j=1,\ldots,n$. Consequently,
\begin{equation*}
\lambda=\left(\frac{\delta^2I}{(2\pi)^dn|2\nu+d|}\right)^{\frac{\nu-\eta}{2\nu+d}}.
\end{equation*}
It is easily checked that
\begin{equation*}
C_\infty^2(\nu,\eta)=\frac1{|2\eta+d|}(n|2\nu+d|)^{\frac{\eta+d/2}{\nu+d/2}}.
\end{equation*}
As a result, $\lambda^2=\beta$. In view of \eqref{EDD}, using calculations similar to those that were above, we obtain
\begin{multline}\label{Es}
E^2_{\infty2}(\Lambda,\mathcal D)\ge\|\Lambda\wx\cd\|^2_{\ld}=\frac{\delta^2}{(2\pi)^d}\int_
{|\psi(\xi)|>\lambda\sqrt{s_2(\xi)}}|\psi(\xi)|^2\,d\xi\\
=\frac{\delta^2}{|2\eta+d|(2\pi)^d}\lambda^{-\frac{2\eta+d}{\nu-\eta}}I=\frac1{(2\pi)^{d\wwg}}
C^2_\infty(\nu,\eta)I^{2/\wwq}\delta^{2\wwg}.
\end{multline}

We estimate the error of the method \eqref{mbe}. Put
\begin{equation*}
a(\xi)=\left(1-\beta\frac{s_2(\xi)}{|\psi(\xi)|^2}\right)_+.
\end{equation*}
Taking the Fourier transform we obtain
\begin{equation*}
\|\Lambda x\cd-\wm(y)\cd\|^2_{\ld}\\
=\frac1{(2\pi)^d}\iRd|\psi(\xi)|^2\left|Fx(\xi)-a(\xi)y(\xi)\right|^2\,d\xi.
\end{equation*}
We set $z\cd=Fx\cd-y\cd$ and note that
\begin{equation*}
\|z\cd\|_{L_\infty(\mathbb R^d)}\le\delta,\quad\frac1{(2\pi)^d}\iRd|\varphi_j(\xi)|^2|Fx(\xi)|^2\,d\xi\le1,\ j=1,\ldots,n.
\end{equation*}
Hence
\begin{equation*}
\|\Lambda x\cd-\wm(y)\cd\|^2_{\ld}
=\frac1{(2\pi)^d}\iRd|\psi(\xi)|^2\left|\left(1-a(\xi)\right)Fx(\xi)+
a(\xi)z(\xi)\right|^2\,d\xi.
\end{equation*}
The integrand can be written as
\begin{equation*}
\left|\frac{|\psi(\xi)|(1-a(\xi))\sqrt{\beta s_2(\xi)}Fx(\xi)}{\sqrt{\beta s_2(\xi)}}
+\sqrt{a(\xi)}\sqrt{a(\xi)}|\psi(\xi)|z(\xi)\right|^2.
\end{equation*}
Using the Cauchy-Bunyakovskii-Schwarz inequality
\begin{equation*}
|ab+cd|^2\le(|a|^2+|c|^2)(|b|^2+|d|^2)
\end{equation*}
we obtain the estimate
$$\|\Lambda x\cd-\wm(y)\cd\|^2_{\ld}
\le\vraisup_{\xi\in\mathbb R^d}S(\xi)\frac1{(2\pi)^d}\iRd\left(\beta s_2(\xi)|Fx(\xi)|^2+
a(\xi)|\psi(\xi)|^2|z(\xi)|^2\right)\,d\xi,$$
where
\begin{equation*}
S(\xi)=\frac{|\psi(\xi)|^2|(1-a(\xi))^2}{\beta s_2(\xi)}+a(\xi).
\end{equation*}
If $|\psi(\xi)|^2\le\beta s_2(\xi)$, then $a(\xi)=0$ and $S(\xi)\le1$. If $|\psi(\xi)|^2>\beta s_2(\xi)$, then $S(\xi)=1$. So we have
\begin{multline*}
e^2_{\infty2}(\Lambda,\mathcal D,\wm)
\le\frac1{(2\pi)^d}\iRd\left(\beta s_2(\xi)|Fx(\xi)|^2+a(\xi)|\psi(\xi)|^2|z(\xi)|^2\right)\,d\xi
\le n\beta\\+\frac{\delta^2}{(2\pi)^d}\int_{|\psi(\xi)|>\lambda\sqrt{s_2(\xi)}}\left(|\psi(\xi)|^2-
\beta s_2(\xi)\right)\,d\xi=n\beta
+\frac{\delta^2}{(2\pi)^d}\int_{|\psi(\xi)|>\lambda\sqrt{s_2(\xi)}}|\psi(\xi)|^2\,d\xi\\
-\beta\frac1{(2\pi)^d}\iRd s_2(\xi)|F\wx(\xi)|^2\,d\xi
=\frac{\delta^2}{(2\pi)^d}\int_
{|\psi(\xi)|>\lambda\sqrt{s_2(\xi)}}|\psi(\xi)|^2\,d\xi\le E^2_{\infty2}(\Lambda,\mathcal D).
\end{multline*}
It follows that the method $\wm(y)\cd$ is optimal. Moreover, by \eqref{Es} we have
\begin{equation*}
E^2_{\infty2}(\Lambda,\mathcal D)=\frac{\delta^2}{(2\pi)^d}\int_{|\psi(\xi)|
>\lambda\sqrt{s_2(\xi)}}|\psi(\xi)|^2\,d\xi=\frac1{(2\pi)^{d\wwg}}C^2_\infty(n,k)I^{2/\wwq}
\delta^{2\wwg}.
\end{equation*}

Similar to the proof of Corollary~\ref{Sdva} we prove that for $p=\infty$ inequality \eqref{ene} is sharp .
\end{proof}

Let $\alpha=(\alpha_1,\ldots,\alpha_d)\in\mathbb R_+^d$. We define the operator $D^\alpha$ (the derivative of order $\alpha$) by
\begin{equation*}
D^\alpha x\cd=F^{-1}((i\xi)^\alpha Fx(\xi))\cd,
\end{equation*}
where $(i\xi)^\alpha=(i\xi_1)^{\alpha_1}\ldots(i\xi_d)^{\alpha_d}$.

Consider problem \eqref{ED} for $D_j=D^{\nu e_j}$, $j=1,\ldots,d$, where $e_j$, $j=1\ldots,d$, is a standard basis in $\mathbb R^d$, and $\Lambda$ defined by \eqref{Lam}. Assume that $\psi\cd$ has the following symmetry property
$$\psi(\ldots,\xi_j,\ldots,\xi_m,\ldots)=\psi(\ldots,\xi_m,\ldots,\xi_j,\ldots),\quad1\le j,m\le d.$$
Moreover, we assume that $\wps\cd$ is continuous function on $\Pi^{d-1}$.

In this case for \eqref{Ieq} and \eqref{Ieq1} we have
\begin{align}
I&=\int_{\Pi^{d-1}}\frac{\wps^{\wwq}(\omega)J(\omega)\,d\omega}{\left(\sum_{k=1}^d\wtt_k^{\,2\nu}(\omega)\right)
^{\wwq(1-\wwg)/2}},\label{I8}\\
I_j'&=\int_{\Pi^{d-1}}\frac{\wps^{\wwq}(\omega)\wtt_j^{\,2\nu}(\omega)J(\omega)\,d\omega}
{\left(\sum_{k=1}^d\wtt_k^{\,2\nu}(\omega)\right)^{\wwq(1-\wwg)/2+1}},\ j=1,\ldots,d.\notag
\end{align}
Similar to how it was done for weights \eqref{wei} we prove that $I<\infty$ and $I_1'=\ldots=I_d'$. Thus, from Theorem~\ref{pred1} we obtain

\begin{corollary}\label{Cor8}
Let $2<p\le\infty$ and $\nu>\eta\ge0$. Then
\begin{equation*}
E_{p2}(\Lambda,(D^{\nu e_1},\ldots,D^{\nu e_d}))=\frac1{(2\pi)^{d\wwg/2}}C_p(\nu,\eta)I^{1/\wwq}\delta^{\wwg},
\end{equation*}
where $I$ is defined by \eqref{I8}. The method
\begin{equation*}
\wm(y)(t)=F^{-1}\left(\left(1-\beta\frac{\sum_{j=1}^d|t_j|^{2\nu}}
{|\psi(t)|^2}\right)_+\psi(t)y(t)\right),
\end{equation*}
where
\begin{equation*}
\beta=\frac{1-\wwg}{d(2\pi)^{d\wwg}}C_p^2(\nu,\eta)\left(\delta I^{1/2-1/p}\right)^{2\wwg},
\end{equation*}
is optimal.

The sharp inequality
$$\|\Lambda x\cd\|_{\ld}\\
\le\frac1{(2\pi)^{d\wwg/2}}C_p(\nu,\eta)I^{1/\wwq}\|Fx\cd\|_{\lpd}^{\wwg}
\left(\max_{1\le j\le d}\|D^{\nu e_j}x\cd\|_{\ld}\right)^{1-\wwg}$$
holds.
\end{corollary}

As functions $\psi\cd$ defining the operator $\Lambda$ we can consider the functions
\begin{equation*}
\psi_\theta(\xi)=(|\xi_1|^\theta+\ldots+|\xi_d|^\theta)^{2/\theta},\quad\theta>0.
\end{equation*}
The corresponding operator is denoted by $\Lambda_\theta$. In particular, $\Lambda_2=-\Delta$, where $\Delta$ is the Laplace operator. We denote by $\Lambda_\theta^{\eta/2}$ the operator $\Lambda$ which is defined by $\psi\cd=\psi_\theta^{\eta/2}\cd$.

Now we consider the case when $p=2$.

\begin{theorem}
Let $\nu>\eta>0$, $\nu\ge1$, and $0<\theta\le2\nu$. Then
\begin{equation}\label{E}
E_{22}(\Lambda_\theta^{\eta/2},(D^{\nu e_1},\ldots,D^{\nu e_d}))=d^{\eta/\theta}\left(\frac{\delta}{(2\pi)^{d/2}}\right)^{1-\eta/\nu},
\end{equation}
and all methods
\begin{equation}\label{mma}
\wm(y)(t)=F^{-1}\left(a(t)\psi_\theta^{\eta/2}(t)y(t)\right),
\end{equation}
where $a\cd$ are measurable functions satisfying the condition
\begin{equation}\label{aa}
\psi_\theta^\eta(\xi)\left(\frac{|1-a(\xi)|^2}{\lambda_2\sum_{j=1}^d|\xi_j|
^{2\nu}}+\frac{|a(\xi)|^2}
{(2\pi)^d\lambda_1}\right)\le1,
\end{equation}
in which
\begin{equation*}
\lambda_1=\frac{d^{2\eta/\theta}}{(2\pi)^d}\left(1-\frac\eta\nu\right)\left(\frac{(2\pi)^d}{\delta^2}
\right)^{\eta/\nu},
\quad\lambda_2=\frac\eta\nu d^{2\eta/\theta-1}\left(\frac{(2\pi)^d}{\delta^2}
\right)^{\eta/\nu-1},
\end{equation*}
are optimal.

The sharp inequality
\begin{equation}\label{Sl2}
\|\Lambda_\theta^{\eta/2}x\cd\|_{\ld}
\le\frac{d^{\eta/\theta}}{(2\pi)^{d(1-\eta/\nu)/2}}\|Fx\cd\|_{\ld}^{\eta/\nu}
\left(\max_{1\le j\le d}\|D^{\nu e_j}x\cd\|_{\ld}\right)^{1-\eta/\nu}
\end{equation}
holds.
\end{theorem}

\begin{proof}
It follows by Lemma~\ref{L1} that
\begin{equation}\label{EDD2}
E_{22}(\Lambda_\theta^{\eta/2},(D^{\nu e_1},\ldots,D^{\nu e_d}))\ge\sup_{\substack{x\cd\in W_2^{(D^{\nu e_1},\ldots,D^{\nu e_d})}\\\|Fx\cd\|_{L_2(\mathbb R^d)}\le\delta}}\|\Lambda_\theta^{\eta/2}x\cd\|_{L_2(\mathbb R^d)}.
\end{equation}
Given $0<\varepsilon<(2\pi)^{d/(2\nu)}\delta^{-1/\nu}$, we set
\begin{equation*}
\wxi_\varepsilon=\left(\frac{(2\pi)^d}{\delta^2}\right)^{\frac1{2\nu}}(1,\ldots,1)-(\varepsilon,
\ldots,\varepsilon),\quad B_\varepsilon=\{\xi\in\mathbb R^d:|\xi-\wxi_\varepsilon|<\varepsilon\,\}.
\end{equation*}
Consider a function $x_\varepsilon\cd$ such that
\begin{equation}\label{xeps}
Fx_\varepsilon(\xi)=\begin{cases}\dfrac\delta{\sqrt{\mes B_\varepsilon}},&\xi\in B_\varepsilon,\\
0,&\xi\notin B_\varepsilon.\end{cases}
\end{equation}
Then $\|Fx_\varepsilon\cd\|^2_{\ld}=\delta^2$ and
\begin{equation*}
\|D^{\nu e_j}x_\varepsilon\cd\|^2_{\ld}
=\frac{\delta^2}{(2\pi)^d\mes B_\varepsilon}\int_{B_\varepsilon}|\xi_j|^{2\nu}\,d\xi
\le1,\ j=1,\ldots,d.
\end{equation*}
By virtue of \eqref{EDD2} we have
\begin{multline*}
E_{22}^2(\Lambda_\theta^{\eta/2},(D^{\nu e_1},\ldots,D^{\nu e_d}))\ge\|\Lambda_\theta^{\eta/2}x_\varepsilon\cd\|_{L_2(\mathbb R^d)}^2\\
=\frac{\delta^2}{(2\pi)^d\mes B_\varepsilon}\int_{B_\varepsilon}\psi_\theta^\eta(\xi)\,d\xi
=\frac{\delta^2}{(2\pi)^d}\psi_\theta^\eta(\widetilde\xi_\varepsilon),\quad
\widetilde\xi_\varepsilon\in B_\varepsilon.
\end{multline*}
Letting $\varepsilon\to0$ we obtain the estimate
\begin{equation}\label{EE}
E_{22}^2(\Lambda_\theta^{\eta/2},(D^{\nu e_1},\ldots,D^{\nu e_d}))\ge d^{2\eta/\theta}\left(\frac{\delta^2}{(2\pi)^d}\right)^{1-\eta/\nu}.
\end{equation}

We will find optimal methods among methods \eqref{mma}. Passing to the Fourier transform we have
\begin{equation*}
\|\Lambda_\theta^{\eta/2}x\cd-\wm(y)\cd\|^2_{\ld}\\
=\frac1{(2\pi)^d}\iRd\psi_\theta^\eta(\xi)\left|Fx(\xi)-a(\xi)y(\xi)\right|^2\,d\xi.
\end{equation*}
We set $z\cd=Fx\cd-y\cd$ and note that
\begin{equation*}
\iRd|z(\xi)|^2\,d\xi\le\delta^2,\quad\frac1{(2\pi)^d}\iRd|\xi_j|^{2\nu}|Fx(\xi)|^2\,d\xi\le1,
\ j=1,\ldots,d.
\end{equation*}
Then
\begin{equation*}
\|\Lambda_\theta^{\eta/2}x\cd-\wm(y)\cd\|^2_{\ld}\\
=\frac1{(2\pi)^d}\iRd\psi_\theta^\eta(\xi)\left|\left(1-a(\xi)\right)Fx(\xi)+
a(\xi)z(\xi)\right|^2\,d\xi.
\end{equation*}
We write the integrand as
\begin{equation*}
\psi_\theta^\eta(\xi)\left|\frac{(1-a(\xi))\sqrt{\lambda_2}\left(\sum_{j=1}^d|\xi_j|
^{2\nu}\right)^{1/2}Fx(\xi)}{\sqrt{\lambda_2}\left(\sum_{j=1}^d|\xi_j|
^{2\nu}\right)^{1/2}}+
\frac{a(\xi)}{(2\pi)^{d/2}\sqrt{\lambda_1}}(2\pi)^{d/2}\sqrt{\lambda_1}z(\xi)\right|^2.
\end{equation*}
Applying the Cauchy-Bunyakovskii-Schwarz inequality we obtain the estimate
\begin{multline*}
\|\Lambda_\theta^{\eta/2}x\cd-\wm(y)\cd\|^2_{\ld}\\
\le\vraisup_{\xi\in\mathbb R^d}S(\xi)\frac1{(2\pi)^d}\iRd\left(\lambda_2\sum_{j=1}^d|\xi_j|
^{2\nu}|Fx(\xi)|^2+
(2\pi)^d\lambda_1|z(\xi)|^2\right)\,d\xi,
\end{multline*}
where
\begin{equation*}
S(\xi)=\psi_\theta^\eta(\xi)\left(\frac{|1-a(\xi)|^2}{\lambda_2\sum_{j=1}^d|\xi_j|
^{2\nu}}+\frac{|a(\xi)|^2}{(2\pi)^d\lambda_1}\right).
\end{equation*}
If we assume that $S(\xi)\le1$ for almost all $\xi$, then taking into account \eqref{EE}, we get
\begin{multline*}
e^2_{22}(\Lambda_\theta^{\eta/2},(D^{\nu e_1},\ldots,D^{\nu e_d}),\wm)
\le\frac1{(2\pi)^d}\iRd\left(\lambda_2\sum_{j=1}^d|\xi_j|
^{2\nu}|Fx(\xi)|^2+
(2\pi)^d\lambda_1|z(\xi)|^2\right)\,d\xi\\
\le\lambda_2d+\lambda_1\delta^2=d^{2\eta/\theta}\left(\frac{\delta^2}{(2\pi)^d}\right)^{1-\eta/\nu}\le E_{22}^2(\Lambda_\theta^{\eta/2},(D^{\nu e_1},\ldots,D^{\nu e_d})).
\end{multline*}
This proves \eqref{E} and shows that the methods under consideration are optimal.

It remains to verify that the set of functions $a\cd$ satisfying \eqref{aa} is nonempty. Put
\begin{equation*}
a(\xi)=\frac{(2\pi)^d\lambda_1}{(2\pi)^d\lambda_1+\lambda_2\sum_{j=1}^d|\xi_j|
^{2\nu}}.
\end{equation*}
Then
\begin{equation*}
S(\xi)=\frac{\psi_\theta^\eta(\xi)}{(2\pi)^d\lambda_1+\lambda_2\sum_{j=1}^d|\xi_j|
^{2\nu}}.
\end{equation*}
Since $\theta\le2\nu$ by H\"older's inequality
\begin{equation*}
\sum_{j=1}^d|\xi_j|^\theta\le\biggl(\sum_{j=1}^d|\xi_j|^{2\nu}\biggr)^{\theta/(2\nu)}
d^{1-\theta/(2\nu)}.
\end{equation*}
Putting $\rho=(|\xi_1|^\theta+\ldots+|\xi_d|^\theta)^{1/\theta}$, we obtain
\begin{equation*}
\sum_{j=1}^d|\xi_j|^{2\nu}\ge\rho^{2\nu}d^{1-2\nu/\theta}.
\end{equation*}
Thus,
\begin{equation*}
S(\xi)\le\frac{\rho^{2\eta}}{(2\pi)^d\lambda_1+\lambda_2\rho^{2\nu}d^{1-2\nu/\theta}}.
\end{equation*}
It is easily checked that the function $f(\rho)=(2\pi)^d\lambda_1+\lambda_2\rho^{2\nu}d^{1-2\nu/\theta}-\rho^{2\eta}$ reaches a minimum on  $[0,+\infty)$ at
\begin{equation*}
\rho_0=d^{1/\theta}\left(\frac{(2\pi)^d}{\delta^2}
\right)^{1/(2\nu)}.
\end{equation*}
Moreover, $f(\rho_0)=0$. Consequently, $f(\rho)\ge0$ for all $\rho\ge0$. Hence $S(\xi)\le1$ for all $\xi$.

Inequality \eqref{Sl2} is proved by the analogy with the proof of Corollary~\ref{Sdva}.
\end{proof}

\subsection{Recovery in the metric $L_\infty(\mathbb R^d)$}

Put
\begin{equation*}
\begin{gathered}
\gamma_1=\frac{\nu-\eta-d/2}{\nu+d(1/2-1/p)},\quad q_1=\frac1{1/2+\gamma_1(1/2-1/p)},\\
\wCC_p(\nu,\eta)=\gamma_1^{-\frac{\gamma_1}p}
\left(\frac{1-\gamma_1}n\right)^{-\frac{1-\gamma_1}2}\Biggl(\frac{B\left(q_1\gamma_1/p+1,
q_1(1-\gamma_1)/2\right)}{2|\nu-\eta-d/2|}\Biggr)^{1/q_1}.
\end{gathered}
\end{equation*}
For $1<p<\infty$ we define $k\cd$ by the equality
\begin{equation*}
\frac{k(t)}{(1-k(t))^{p-1}}=(2\pi)^d\frac{|\psi(t)|^{p-2}}{s_2^{p-1}(t)}.
\end{equation*}
We set
\begin{equation*}
k(t)=\begin{cases}\displaystyle
\min\left\{1,(2\pi)^d|\psi(t)|^{-1}\right\},&p=1,\\
\left(1-s_2(t)|\psi(t)|^{-1}\right)_+,&p=\infty.\end{cases}
\end{equation*}

\begin{theorem}\label{pred11}
Let $1\le p\le\infty$, $\gamma_1\in(0,1)$. Assume that
\begin{equation*}
I=\int_{\Pi^{d-1}}\frac{\wps^{q_1}(\omega)}{\wc_2^{\,q_1(1-\gamma_1)/2}(\omega)}J(\omega)
\,d\omega<\infty
\end{equation*}
and $I'_1=\ldots=I'_n$, where
\begin{equation*}
I'_j=\int_{\Pi^{d-1}}\frac{\wps^{q_1}(\omega)\wva_j^2(\omega)}
{\wc_2^{\,q_1(1-\gamma_1)/2+1}(\omega)}J(\omega)\,d\omega,\ j=1,\ldots,n.
\end{equation*}
Then
\begin{equation*}
E_{p\infty}(\Lambda,\mathcal D)= \frac1{(2\pi)^{d(1+\gamma_1)/2}}\wCC_p(\nu,\eta)I^{1/q_1}\delta^{\gamma_1}.
\end{equation*}
The method
\begin{equation*}
\wm(y)(t)=F^{-1}\left(k\left(\xi_1^{\frac1{n+d(1/2-1/p)}}t\right)\psi(t)y(t)\right),
\end{equation*}
where
\begin{equation*}
\xi_1=\delta\gamma_1^{-\frac{q_1}{2p}}\left(\frac{(1-\gamma_1)
\wCC_p(\nu,\eta)I^{1/q_1}}{n(2\pi)^{d(1+\gamma_1)/2}}\right)^{q_1(1/2-1/p)},
\end{equation*}
is optimal.

The sharp inequality
\begin{equation}\label{ene1}
\|\Lambda x\cd\|_{L_\infty(\mathbb R^d)}
\le\frac1{(2\pi)^{d(1+\gamma_1)/2}}\wCC_p(\nu,\eta)I^{1/q_1}\|Fx\cd\|_{\lpd}^{\gamma_1}
\left(\max_{1\le j\le n}\|D_jx\cd\|_{\ld}\right)^{1-\gamma_1}
\end{equation}
holds.
\end{theorem}

\begin{proof}
Using an estimate similar to \eqref{EDD} we have
\begin{equation*}
E_{p\infty}(\Lambda,\mathcal D)\ge\sup_{\substack{x\cd\in W_p^{\mathcal D}\\\|Fx\cd\|_{\lpd}\le\delta}}\|\Lambda x\cd\|_{L_\infty(\mathbb R^d)}.
\end{equation*}
Assume that $x\cd\in W_p^{\mathcal D}$ and $\|Fx\cd\|_{\lpd}\le\delta$. If $\wx\cd$ is such that
$F\wx(\xi)=\varepsilon(\xi)e^{-i\la t,\xi\ra}Fx(\xi)$,
where
\begin{equation*}
\varepsilon(\xi)=\begin{cases}
\dfrac{\ov{\psi(\xi)}\ov{Fx(\xi)}}{|\psi(\xi)Fx(\xi)|},&\psi(\xi)Fx(\xi)\ne0,\\
0,&\psi(\xi)Fx(\xi)=0,
\end{cases}
\end{equation*}
then we obtain $\wx\cd\in W_p^{\mathcal D}$, $\|F\wx\cd\|_{\lpd}\le\delta$ and
\begin{equation*}
\biggl|\int_{\mathbb R^d}\psi(\xi)F\wx(\xi)e^{i\la t,\xi\ra}\,d\xi\biggr|=\int_{\mathbb R^d}|\psi(\xi)Fx(\xi)|\,d\xi.
\end{equation*}
Hence
\begin{equation}\label{Ein}
E_{p\infty}(\Lambda,\mathcal D)\ge\frac1{(2\pi)^d}\sup_{\substack{x\cd\in W_p^{\mathcal D}\\\|Fx\cd\|_{\lpd}\le\delta}}\int_{\mathbb R^d}|\psi(\xi)Fx(\xi)|\,d\xi.
\end{equation}

Let $1\le p<\infty$. It follows from \eqref{Edv} that
\begin{equation*}
E_{p\infty}(\Lambda,\mathcal D)\ge E(p,1,2),
\end{equation*}
where, in the problem of the evaluation of $E(p,1,2)$, the functions $\varphi_j\cd$ should be
replaced by the function $(2\pi)^{-d/2}\varphi_j\cd$, and the function $\psi\cd$ by $(2\pi)^{-d}\psi\cd$. From Theorem~\ref{T3} we obtain
\begin{equation*}
E_{p\infty}(\Lambda,\mathcal D)\ge\frac1{(2\pi)^{d(1+\gamma_1)/2}}K\delta^{\gamma_1},
\end{equation*}
where
\begin{equation*}
K=\gamma_1^{-\frac{\gamma_1}p}\left(\frac{1-\gamma_1}n\right)^{-\frac{1-\gamma_1}2}
\Biggl(\frac{B\left(q_1\gamma_1/p,q_1(1-\gamma_1)/2\right)I}{|\nu+d(1/2-1/p)|(2\gamma_1 +(1-\gamma_1)p)}\Biggr)^{1/q_1}.
\end{equation*}
From the properties of the beta-function
\begin{multline*}
\frac{B\left(q_1\gamma_1/p,q_1(1-\gamma_1)/2\right)}{|\nu+d(1/2-1/p)|(2\gamma_1 +(1-\gamma_1)p)}\\=
\frac{B\left(q_1\gamma_1/p+1,q_1(1-\gamma_1)/2\right)(q_1\gamma_1/p+q_1(1-\gamma_1)/2)}
{|\nu+d(1/2-1/p)|(2\gamma_1 +(1-\gamma_1)p)q_1\gamma_1/p}\\
=\frac{B\left(q_1\gamma_1/p+1,q_1(1-\gamma_1)/2\right)}{2|\nu-\eta-d/2|}.
\end{multline*}
Thus,
\begin{equation*}
E_{p\infty}(\Lambda,\mathcal D)\ge \frac1{(2\pi)^{d(1+\gamma_1)/2}}\wCC_p(\nu,\eta)I^{1/q_1}\delta^{\gamma_1}.
\end{equation*}
Moreover, it follows from the same Theorem~\ref{T3} that
\begin{equation*}
\int_{\mathbb R^d}\left|\frac1{(2\pi)^d}\psi(\xi)F(\xi)-m(y)(\xi)\right|\,d\xi\le E(p,1,2),
\end{equation*}
where
\begin{equation*}
m(y)(t)=\frac1{(2\pi)^d}k\left(\xi_1^{\frac1{\nu+d(1/2-1/p)}}t\right)\psi(t)y(t),
\end{equation*}
and
\begin{multline*}
\xi_1=\frac\delta{\gamma_1^{1/p}}\left(\frac{1-\gamma_1}n\right)^{1/2}
\Biggl(\frac{B\left(q_1\gamma_1/p,q_1(1-\gamma_1)/2\right)I(2\pi)^{-dq_1(1+\gamma_1)/2}}{|\nu+d(1/2-1/p)|(2\gamma_1 +(1-\gamma_1)p)}\Biggr)^{1/2-1/p}\\
=\delta\gamma_1^{-\frac{q_1}{2p}}\left(\frac{(1-\gamma_1)
\wCC_p(\nu,\eta)I^{1/q_1}}{n(2\pi)^{d(1+\gamma_1)/2}}\right)^{q_1(1/2-1/p)}.
\end{multline*}
Consequently,
\begin{multline*}
\biggl|\frac1{(2\pi)^d}\int_{\mathbb R^d}\psi(\xi)F(\xi)e^{i\la t,\xi\ra}\,d\xi-\int_{\mathbb R^d}m(y)(\xi)e^{i\la t,\xi\ra}\,d\xi\biggr|\\
\le\int_{\mathbb R^d}\left|\frac1{(2\pi)^d}\psi(\xi)F(\xi)-m(y)(\xi)\right|\,d\xi\le E(p,1,2)
\le E_{p\infty}(\Lambda,\mathcal D).
\end{multline*}
It follows that the method $\wm(y)\cd$ is optimal, and the error of optimal recovery
coincides with $E(p,1,2)$.

Now we consider the case when $p=\infty$. Put
\begin{equation*}
s(\xi)=\begin{cases}\dfrac{\psi(\xi)}{|\psi(\xi)|},&\psi(\xi)\ne0,\\
1,&\psi(\xi)=0.\end{cases}
\end{equation*}
Let $\wx\cd$ be such that
\begin{equation*}
F\wx(\xi)=\begin{cases}\delta\ov{s(\xi)},&|\psi(\xi)|\ge\lambda s_2(\xi),\\
\dfrac{\delta\ov{\psi(\xi)}}{\lambda s_2(\xi)},&|\psi(\xi)|<\lambda s_2(\xi).
\end{cases}
\end{equation*}

We choose $\lambda>0$ such that
\begin{equation*}
\frac1{(2\pi)^d}\iRd|\varphi_j(\xi)|^2|F\wx(\xi)|^2\,d\xi=1,\quad j=1,\ldots,n.
\end{equation*}
Now, to find $\lambda$ we have the equation
\begin{equation*}
\frac{\delta^2}{(2\pi)^d}\int_{|\psi(\xi)|\ge\lambda s_2(\xi)}|\varphi_j(\xi)|^2\,d\xi
+\frac{\delta^2\lambda^{-2}}{(2\pi)^d}\int_{|\psi(\xi)|<\lambda s_2(\xi)}\frac{|\varphi_j(\xi)|^2|\psi(\xi)|^2
}{s_2^2(\xi)}\,d\xi=1.
\end{equation*}
If $\nu>\eta+d/2$, then from the fact that $\gamma_1\in(0,1)$ it follows that $\eta>-d$. In this case it is easy to check that $2\nu>\eta$ and $2\nu+d>0$. Passing to the polar transformation we obtain
\begin{multline*}
\frac{\delta^2}{(2\pi)^d}\int_{\Pi_{d-1}}\wva^2_j(\omega)J(\omega)\,d\omega\int_0^{\Phi_2(\omega)}
\rho^{2\nu+d-1}\,d\rho\\
+\frac{\delta^2\lambda^{-2}}{(2\pi)^d}\int_{\Pi_{d-1}}\frac{\wva^2_j(\omega)\wps^2(\omega)}
{\wc_2^2(\omega)}
J(\omega)\,d\omega\int_{\Phi_2(\omega)}^{+\infty}\rho^{-2\nu+2\eta+d-1}\,d\rho=1,
\end{multline*}
where
\begin{equation*}
\Phi_2(\omega)=\left(\frac{\wps(\omega)}{\lambda\wc_2(\omega)}\right)
^{\frac1{2\nu-\eta}}.
\end{equation*}
Thus,
\begin{equation*}
\frac{\delta^2}{(2\pi)^d}\lambda^{-\frac{2\nu+d}{2\nu-\eta}}\frac{4\nu-2\eta}
{(2\nu+d)(2\nu-2\eta-d)}I_j=1.
\end{equation*}

If $\nu<\eta+d/2$, then it follows from $\gamma_1\in(0,1)$ that $\eta<-d$, $2\nu<\eta$, and $2\nu+d<0$. Passing to the polar transformation we obtain
\begin{multline*}
\frac{\delta^2}{(2\pi)^d}\int_{\Pi_{d-1}}\wva^2_j(\omega)J(\omega)\,d\omega
\int_{\Phi_2(\omega)}^{+\infty}\rho^{2\nu+d-1}\,d\rho\\
+\frac{\delta^2\lambda^{-2}}{(2\pi)^d}\int_{\Pi_{d-1}}\frac{\wva^2_j(\omega)\wps^2(\omega)}
{\wc_2^2(\omega)}J(\omega)\,d\omega\int_0^{\Phi_2(\omega)}\rho^{-2\nu+2\eta+d-1}\,d\rho=1,
\end{multline*}
For this case we have
\begin{equation*}
\frac{\delta^2}{(2\pi)^d}\lambda^{-\frac{2\nu+d}{2\nu-\eta}}\frac{2\eta-4\nu}
{(2\nu+d)(2\nu-2\eta-d)}I_j=1.
\end{equation*}
Combining both of these cases and taking into account that $I_j=I/n$, $j=1,\ldots,n$, we get
\begin{equation*}
\lambda=\left(\frac{2\delta^2|2\nu-\eta|I}{(2\pi)^dn(2\nu+d)(2\nu-2\eta-d)}\right)^
{\frac{2\nu-\eta}{2\nu+d}}.
\end{equation*}

It follows by \eqref{Ein} that
\begin{multline*}
E_{\infty\infty}(\Lambda,\mathcal D)\ge\frac1{(2\pi)^d}\int_{\mathbb R^d}|\psi(\xi)F\wx(\xi)|\,d\xi=\frac{\delta}{(2\pi)^d}\int_{|\psi(\xi)|\ge\lambda
s_2(\xi)}|\psi(\xi)|\,d\xi\\
+\frac{\delta}{\lambda(2\pi)^d}\int_{|\psi(\xi)|<\lambda
s_2(\xi)}\frac{|\psi(\xi)|^2
}{s_2(\xi)}\,d\xi.
\end{multline*}
Using calculations similar to those that were above, we obtain
\begin{equation*}
E_{\infty\infty}(\Lambda,\mathcal D)\ge\frac{\delta|2\nu-\eta|\lambda^{-\frac{\eta+d}{2\nu-\eta}}I}{(2\pi)^d(\eta+d)(2\nu-2\eta-d)}\\
=E_0,
\end{equation*}
where
\begin{equation*}
E_0=\frac{(n|\nu+d/2|)^{\frac{\eta+d}{2\nu+d}}}{\eta+d}\left(\frac{(2\nu-\eta)I}{(2\pi)^d
(2\nu-2\eta-d)}\right)^{\frac{2\nu-\eta}{2\nu+d}}\delta^{\frac{2\nu-2\eta-d}{2\nu+d}}.
\end{equation*}

We prove that for all $x\cd\in X_\infty$ the equality
\begin{multline}\label{Toz}
\Lambda x(t)=\frac1{(2\pi)^d}\int_{|\psi(\xi)|\ge\lambda
s_2(\xi)}\left(\psi(\xi)-\lambda s(\xi)s_2(\xi)\right)Fx(\xi)e^{i\la t,\xi\ra}\,d\xi\\
+\frac\lambda{\delta(2\pi)^d}\int_{\mathbb R^d}s_2(\xi)Fx(\xi)\ov{F\wx(\xi)}e^{i\la t,\xi\ra}\,d\xi.
\end{multline}
holds. Indeed,
\begin{multline*}
\frac1{(2\pi)^d}\int_{|\psi(\xi)|\ge\lambda
s_2(\xi)}\left(\psi(\xi)-\lambda s(\xi)s_2(\xi)\right)Fx(\xi)e^{i\la t,\xi\ra}\,d\xi\\
+\frac\lambda{\delta(2\pi)^d}\int_{\mathbb R^d}s_2(\xi)Fx(\xi)\ov{F\wx(\xi)}e^{i\la t,\xi\ra}\,d\xi\\
=\frac1{(2\pi)^d}\int_{|\psi(\xi)|\ge\lambda
s_2(\xi)}\left((\psi(\xi)-\lambda s(\xi)s_2(\xi)\right)Fx(\xi)e^{i\la t,\xi\ra}\,d\xi\\
+\frac1{(2\pi)^d}\int_{|\psi(\xi)|\ge\lambda
s_2(\xi)}
\lambda s(\xi)s_2(\xi)Fx(\xi)e^{i\la t,\xi\ra}\,d\xi
+\frac1{(2\pi)^d}\int_{|\psi(\xi)|<\lambda
s_2(\xi)}\psi(\xi)Fx(\xi)e^{i\la t,\xi\ra}\,d\xi\\
=\frac1{(2\pi)^d}\int_{\mathbb R^d}\psi(\xi)Fx(\xi)e^{i\la t,\xi\ra}\,d\xi=\Lambda x(t).
\end{multline*}

We estimate the error of the method
\begin{equation*}
m(y)(t)=\frac1{(2\pi)^d}\int_{|\psi(\xi)|\ge\lambda
s_2(\xi)}\left(\psi(\xi)-\lambda s(\xi)s_2(\xi)\right)y(\xi)e^{i\la t,\xi\ra}\,d\xi.
\end{equation*}
We have
\begin{multline*}
|\Lambda x(t)-m(y)(t)|\le\biggl|\frac1{(2\pi)^d}\int_{\mathbb R^d}\psi(\xi)Fx(\xi)e^{i\la t,\xi\ra}\,d\xi\biggr.\\
\biggl.-\frac1{(2\pi)^d}\int_{|\psi(\xi)|\ge\lambda
s_2(\xi)}\left(\psi(\xi)-\lambda s(\xi)s_2(\xi)\right)Fx(\xi)e^{i\la t,\xi\ra}\,d\xi\biggr|\\
+\frac1{(2\pi)^d}\int_{|\psi(\xi)|\ge\lambda
s_2(\xi)}\left|\psi(\xi)-\lambda s(\xi)s_2(\xi)\right||Fx(\xi)-y(\xi)|\,d\xi.
\end{multline*}
If $x\cd$ such that
\begin{equation*}
\|Fx\cd-y\cd\|_{L_\infty(\mathbb R^d)}\le\delta,\quad\frac1{(2\pi)^d}\iRd|\varphi_j(\xi)|^2|Fx(\xi)|^2\,d\xi\le1,\ j=1,\ldots,n,
\end{equation*}
then, taking into account \eqref{Toz}, we obtain
\begin{equation*}
|\Lambda x(t)-m(y)(t)|\le\frac\lambda{\delta(2\pi)^d}\int_{\mathbb R^d}s_2(\xi)|Fx(\xi)||F\wx(\xi)|\,d\xi+\mu\le\frac{n\lambda}\delta+\mu,
\end{equation*}
where
\begin{equation*}
\mu=\frac\delta{(2\pi)^d}\int_{|\psi(\xi)|\ge\lambda
s_2(\xi)}\left(|\psi(\xi)|-\lambda s_2(\xi)\right)\,d\xi.
\end{equation*}
Passing to the polar transformation we find
\begin{gather*}
\frac\delta{(2\pi)^d}\int_{|\psi(\xi)|\ge\lambda
s_2(\xi)}
|\psi(\xi)|\,d\xi=
\frac{\delta\lambda^{-\frac{\eta+d}{2\nu-\eta}}}{(2\pi)^d|\eta+d|}I,\\
\frac{\delta\lambda}{(2\pi)^d}\int_{|\psi(\xi)|\ge\lambda
s_2(\xi)}
s_2(\xi)\,d\xi=\frac{\delta\lambda^{-\frac{\eta+d}{2\nu-\eta}}}
{(2\pi)^d|2\nu+d|}I.
\end{gather*}
Hence
\begin{equation*}
\mu=\frac{\delta\lambda^{-\frac{\eta+d}{2\nu-\eta}}|2\nu-\eta|}
{(2\pi)^d(\eta+d)(2\nu+d)}I.
\end{equation*}
It is easily checked that $n\lambda/\delta+\mu=E_0$, and therefore
\begin{equation*}
e_{\infty\infty}(\Lambda,\mathcal D,m)\le E_0\le E_{\infty\infty}(\Lambda,\mathcal D).
\end{equation*}
It follows that $m(y)\cd$ is an optimal method, and the error of optimal recovery
is $E_0$. It is easily checked that for $p=\infty$
\begin{equation*}
\frac1{(2\pi)^{d(1+\gamma_1)/2}}\wCC_\infty(\nu,\eta)I^{1/q_1}\delta^{\gamma_1}=E_0.
\end{equation*}

We evaluate $\xi_1$ for $p=\infty$. We have
\begin{equation}\label{xx}
\xi_1=\delta\left(\frac{(1-\gamma_1)
\wCC_\infty(\nu,\eta)I^{1/q_1}}{n(2\pi)^{d(1+\gamma_1)/2}}\right)^{q_1/2}=
\lambda^{\frac{\nu+d/2}{2\nu-\eta}}.
\end{equation}
The method $m(y)\cd$ can be written as
\begin{equation*}
m(y)(t)=F^{-1}\left(\left(1-\lambda\frac{s_2(\xi))}{|\psi(t)|}\right)_+
\psi(t)y(t)\right).
\end{equation*}
In view of \eqref{xx} we have
\begin{equation*}
m(y)(t)=F^{-1}\left(k\left(\xi_1^{\frac1{n+d/2}}t\right)\psi(t)y(t)\right)=\wm(y)(t).
\end{equation*}

Inequality \eqref{ene1} is proved by the analogy with the proof of Corollary~\ref{Sdva}.
\end{proof}

It is not difficult to formulate a corollary from Theorem~\ref{pred11} analogous to Corollary~\ref{Cor8} for the same $\Lambda$ and $\mathcal D=(D^{\nu e_1},\ldots,D^{\nu e_d})$.

\end{document}